\documentclass[bj,noshowframe]{imsart}

\usepackage{amsfonts}
\usepackage{amsmath,bm}
\usepackage{amssymb,mathrsfs,nicefrac,graphicx}
\usepackage[mathscr]{eucal} 
\usepackage{amsthm,nicefrac,bigints}
\usepackage[utf8]{inputenc}
\usepackage{epsfig}
\usepackage{multirow}
\usepackage{eurosym}
\usepackage{enumerate}
\usepackage{enumitem}
\usepackage{placeins}

\usepackage{bbm,bm}

\startlocaldefs

\newtheorem{theorem}{Theorem} 
\newtheorem{proposition}[theorem]{Proposition}
\newtheorem{lemma}[theorem]{Lemma}

\theoremstyle{remark}\newtheorem{remark}[theorem]{Remark}
\theoremstyle{definition}
\theoremstyle{definition}

\newcommand{\bN}{\mathbb{N}}
\newcommand{\bP}{\mathbb{P}}

\newcommand{\bR}{\mathbb{R}}

\newcommand{\bH}{\mathbb{H}}

\newcommand{\bS}{\mathbb{S}}

\newcommand{\cL}{\mathcal{L}}

\newcommand{\ra}{\text{\rm ra}}

\newcommand{\var}{\text{\rm var}}
\newcommand{\cov}{\text{\rm cov}}

\newcommand{\Exp}{\text{\rm Exp}}

\newcommand{\todistr}{\to_{\text{\rm\tiny distr}}}
\newcommand{\toweak}{\to_{\text{\rm\tiny w}}}
\newcommand{\eqdistr}{=_{\text{\rm\tiny distr}}}

\newcommand{\gf}{^{\text{\rm\tiny GF}}}
\newcommand{\ts}{^{\text{\rm\tiny TS}}}
\newcommand{\sts}{^{\text{\rm\tiny *TS}}}
\newcommand{\sy}{^{\text{\rm\tiny SY}}}
\newcommand{\tr}{^{\text{\rm\tiny T}}}

\newcommand{\pit}{^{\text{\hspace*{0.2mm}\rm\tiny P}}}

\newcommand{\ba}{^{\text{\hspace*{0.2mm}\rm\tiny B}}}

\newcommand{\sle}{\scalebox{0.5}{$\boldsymbol{\le}$}}

\allowdisplaybreaks

\newcommand{\abs}[1]{\lvert#1 \rvert}
\newcommand{\nf}{\nicefrac}

\newcommand{\mc}[2]{\begin{tabular}{@{}c@{}}  {#1} \\  {#2} \end{tabular}} 

\begin{document}

\begin{frontmatter}

\title{Pattern-based tests for two-dimensional copulas}

\runtitle{Tests for copulas}

\begin{aug}

\author{\fnms{Ludwig} \snm{Baringhaus}\ead[label=e1]{lbaring@stochastik.uni-hannover.de}}
\author{\fnms{Rudolf} \snm{Gr{\"u}bel}\ead[label=e2]{rgrubel@stochastik.uni-hannover.de}}
\address[]{ Leibniz Universit\"at Hannover, Institut f{\"u}r Versicherungs- und Finanzmathematik, Postfach 6009, D-30060
  Hannover, Germany\\ \printead{e1,e2}}

\runauthor{L. Baringhaus and R. Gr{\"u}bel}

\end{aug}

\begin{abstract}  
In statistics permutations typically arise in the context of
rank plots for two-dimensional data. Such plots can also be interpreted as
discrete copulas.  In discrete mathematics, typically in the context of 
the description of large (non-random) objects, 
two-dimensional copulas appear as limits
of  permutations and are then known as permutons if the topology refers
to the convergence of pattern frequencies. 
We obtain a functional central limit theorem for such pattern 
frequencies in the context of two-dimensional random samples.  
The result serves as the basis for nonparametric goodness-of-fit tests,
for two-sample tests, and for tests of symmetry. This includes a suitable
variant of the bootstrap for obtaining critical values. 
Pattern-based procedures
are also of interest in a parametric context. We consider two examples,  
the Farlie-Gumbel-Morgenstern class and a family of
delay copulas. We discuss implementation aspects of the resulting procedures and we
provide a simulation study that supplements the theoretical results in the 
nonparametric case.
\end{abstract}

\begin{keyword}
  \kwd{Bootstrap}
  \kwd{copula}
  \kwd{Cram\'er-von Mises type test}
  \kwd{goodness-of-fit test} 
  \kwd{Kolmogorov-Smirnov type test}
  \kwd{permuton}
  \kwd{random permutation}
  \kwd{rank plot}
  \kwd{two-sample test}
\end{keyword}

\end{frontmatter}

\maketitle

\section{Introduction}\label{sec:intro}
Given a two-dimensional data set without ties in the separate sets of 
$x$- and $y$-coordinates there 
is a unique permutation $\pi$ that connects the ranks of the $x$- and
$y$-parts, and the associated rank plot may be seen as the plot of $\pi$.
Patterns in $\pi$ are similarly connected to subsets of the data set. Suppose that the
data arise from a sample $Z_i=(X_i,Y_i)$, $i=1,\ldots,n$, of two-dimensional random 
variables from some distribution $\mu$. If the marginal distribution functions associated 
with $\mu$ are continuous  then the no-ties condition
is satisfied with probability~1, and $\pi$ is the value of the associated random 
permutation $\Pi_n$. We are concerned with the use of pattern counts and their 
linear combinations, which we will refer to as \emph{linear pattern statistics},
as the basis for inference on $\mu$. If such a  statistic $L$ uses patterns of length at most $k$
then we call $k$ the \emph{degree} of $L$.

Kendall's famous nonparametric test for independence of the marginal variables 
is of this type, with degree two. As there are only two patterns of this
length its test statistic $\tau$ can be written as a function of the number of concordant pairs, where
an increase in the $x$-values corresponds to an increase in the $y$-values.
This is  the number of 12-patterns in the permutation, with 12 short 
for the identity permutation of size two. Longer patterns  were already considered
by Hoeffding~\cite{Hoeffding} who showed that there is  a linear 
pattern statistic of degree five  which leads to a test that detects dependence asymptotically 
whenever the underlying distribution has continuous joint and marginal densities.
Yanagimoto~\cite{Yana} proved that this consistency can already be obtained with a linear
pattern statistic of degree four, and  under the weaker condition that $\mu$ has a continuous 
distribution function. Yanagimoto's result has later been rediscovered by Bergsma and Dassios~\cite{BergsmaDassios}, who also considered discrete distributions. The resulting 
Bergsma-Dassios-Yanagimoto (BDY) test has received attention not only in statistics
but also in computer science and discrete mathematics; 
see~\cite{DrtonHanShi,EZLeng,Chan} and the references  given there.

Independence can be translated into the hypothesis that $\mu$ is the product of its marginals
$\mu_X$ and $\mu_Y$.
Naturally the question arises whether linear pattern statistics can similarly be used
in other situations, for example for tests of symmetry, with the hypothesis that $\mu$
is invariant under permutations of the $x$- and $y$-coordinates.  In view of the no-ties
condition  both hypotheses are  invariant under quantile transformations of
the coordinates, which provides a connection to copulas:  
The points of the rank plot for a data set of size $n$ are a set of pairs of integer values 
$(i,j)$ with the property that each $i$ and each $j$, $1\le i,j\le n$, appears exactly
once. The discrete uniform distribution on these pairs therefore has discrete
uniform marginals, and may be interpreted as a discrete copula, or as an empirical
copula  $C_n$ if rescaled by $1/n$, and  the convergence of $C_n$ to the copula $C$ 
associated with $\mu$ is of interest. A standard approach regards $C_n$ and $C$  as functions
on the unit square, and  related  Glivenko-Cantelli type 
results and functional central limit theorems have been investigated by several authors;
see e.g.~\cite{GaensslerStute} and \cite{VaartWellner}. If the empirical process associated
with the original data is the main starting point then $C$ is required to satisfy 
some smoothness condition; see~\cite{Ferma}.  

We approach the underlying distribution 
on the unit square not through its distribution function but through its pattern probabilities, 
in rough analogy to the representation of a distribution on the unit interval by its moments; see also~\cite{GrMet} for a general  discussion. This is related to a current topic in discrete mathematics, 
specifically the investigation of large structures and associated  limit concepts. For graphs, this has led
to a topology based on subgraph frequencies and the characterization of the limits
as graphons; see the research monograph~\cite{Lov}. 
Hoppen et al.~\cite{Hopp} developed an analogue for permutations, based on 
convergence of pattern frequencies, and described the possible limits as
\emph{permutons}, which in fact are the same as two-dimensional copulas. Through pattern 
counting, a permutation appears as a function on the set of permutations rather 
than as the support of a discrete distribution with an empirical copula as its distribution function. 
A corresponding functional central limit theorem
can then be obtained without additional smoothness assumptions. 
Important in this context is the relation to $U$-statistics, already used by Hoeffding in~\cite{Hoeffding}, together with a decomposition that now bears his name. 

Thus, the statistical model underlying the sequence $(\Pi_n)_{n\in\bN}$ of random
permutations is specified by a two-dimensional copula $C$. 
Many families of such copulas have been introduced and discussed in the literature.
Our own interest in this circle of ideas arose in the context
of delay models, see~\cite{BaGrDelDat}. In a queuing context, for example, random
vectors $Z_i=(X_i,Y_i)$, $i=1,\ldots,n$, may appear as arrival and departure times,
with $Y_i=X_i+D_i$ and $D_i\ge 0$ the service times, and we observe the permutation 
that connects the respective orders of arrivals and departures. The resulting 
\emph{delay copulas} will appear repeatedly below.

In Section~\ref{sec:cop} we introduce some basic notation, and  
in Section~\ref{sec:FCLT} we prove the functional central limit theorem for pattern frequencies.
This is then used in Section~\ref{sec:nonpartests}  to obtain nonparametric goodness-of-fit tests, 
two-sample tests,  and tests for symmetry that asymptotically attain the given 
level and are consistent against all alternatives. Such tests have been
investigated earlier in \cite{JASAsymm,FermaGOF,Remillard}, with test statistics
based directly on the empirical copulas. None of the resulting procedures is distribution 
free; indeed,  the limit distribution of the test statistics depends on the
relevant hypothesis in a complicated way, especially so in the context of the two-sample
problem. We develop a suitable variant of the bootstrap that can be used to obtain
critical values that are asymptotically correct. 
A pattern-based approach is also of interest in the parametric context. In Section~\ref{sec:parametric} 
we consider two families in some detail, the Farlie-Gumbel-Morgenstern copulas, and delay copulas 
with exponential delay distribution, as introduced in~\cite{BaGrDelDat}. 
A drawback of the pattern-based procedures is the enormous amount of computation needed, 
even for the calculation of the test statistics.  In Section~\ref{sec:implementation} we address 
implementation aspects, and we report the results of a simulation study for the special case of testing for independence.

\section{Basic concepts and notation}\label{sec:cop}

Let $\bS_n$ be the set of permutations $\pi$ of the set $[n]:=\{1,\ldots,n\}$.  
We use the one-line notation $(\pi(1),\ldots,\pi(n))$, which emphasizes the order 
isomorphism aspect, and we write 
$\bS:=\bigsqcup_{n\in\bN}\bS_n$ for the set of all (finite) permutations,
with $|\sigma|:=k$ if $\sigma\in\bS_k$. On occasion we will further compactify 
the notation, for example by writing 1 for the single element of $\bS_1$, 
or 541362 for $\pi=(5,4,1,3,6,2)$.
For $\sigma \in\bS_m$, $\pi\in\bS_n$ with $n\ge m$, and 
$A=\{i_1,\ldots,i_m\}$ with $1\le i_1<\cdots<i_m\le n$, 
we say that $\sigma$ occurs at $A$ as a pattern in $\pi$
if $\sigma(j)<\sigma(k)$ and $\pi(i_j)<\pi(i_k)$ are equivalent for all $j,k\in[m]$,
and then write $\pi(A)=\sigma$.
Thus $\sigma$ provides an order isomorphism between $[m]$ and the restriction 
$(\pi(i_1),\ldots,\pi(i_m))$ of $\pi$ to~$A$. 
The (relative) \emph{pattern frequency} of $\sigma$ in $\pi$ is
\begin{equation}\label{eq:defT}
 t(\pi,\sigma) \, := \, \frac{\# \{A\subset [n]:  
                            \#A=m,\,\pi(A)=\sigma\}}{\binom{n}{m}}
\end{equation} 
if $n\ge m$; we augment this by putting $t(\pi,\sigma):=0$ if $n<m$.
If, for example, if $\pi=541362$
and $\sigma=231$ then $\pi(A)=\sigma$ for $A$ equal to $\{1,5,6\}$, $\{2,5,6\}$
or $\{4,5,6\}$, so that $t(\pi,\sigma)=3/20$.

The value $t(\pi,\sigma)$ may be interpreted as the probability of observing $\sigma$ 
in $\pi$ if $A\subset [n]$ with $\#A=m$ is chosen uniformly at random.
The sampling interpretation also implies an important connection between
pattern frequencies,
\begin{equation}\label{eq:back}
  t(\pi,\sigma) =\sum_{\rho\in\bS_l} t(\pi,\rho) \, t(\rho,\sigma)
           \quad\text{for all }k\le l\le n,\, \pi\in\bS_n,\sigma\in\bS_k.
\end{equation}
Indeed these are the Chapman-Kolmogorov equations for a suitably chosen Markov chain running 
backwards in time; see~\cite{GrMet}. Considering all pattern densities simultaneously 
we may regard any $\pi\in\bS$ 
as a function on $\bS$ via $\pi \mapsto t(\pi,\sigma)$, $\sigma\in\bS$. This
embeds $\bS$ injectively into the linear space $\bR^\bS$ of real-valued functions on $\bS$.

A tuple $(x_1,\ldots,x_n)$ of pairwise distinct real numbers defines
a permutation $\ra_x\in\bS_n$ via $\ra_x =(r_1,\ldots,r_n)$, where 
$r_k:=\#\{i\in [n]: x_i\le x_k\}$ is the rank of $x_k$ in the data set.
Note that the inverse $\tau=\ra_x^{-1}$ leads to the associated order statistics
in the sense that 
$x_{\tau(1)}<\cdots < x_{\tau(n)}$. Suppose now that $\mu$ is a probability 
measure on (the Borel subsets) of $\bR^2$
with continuous marginal distribution functions $F_1,F_2$, where 
$F_1(x)=\mu((-\infty,x]\times \bR)$ and $F_2(y)=\mu(\bR\times (-\infty,y])$
for all $x,y\in\bR$. Let $(Z_n)_{n\in\bN}$ with $Z_n=(X_n,Y_n)$ be a sequence
of independent two-dimensional random vectors with distribution $\mu$. 
As $F_1$ and $F_2$ are continuous 
we may assume that there are no ties in the component sequences $(X_n)_{n\in\bN}$ 
and $(Y_n)_{n\in\bN}$. Thus, for each $n\in\bN$
there exists a unique $\pi=\Pi_n\in\bS_n$ with the property that $X_i<X_j$ and
$Y_{\pi(i)}<Y_{\pi(j)}$ are equivalent for $i,j\in[n]$. The permutation 
$\pi$ connects the two orders induced on $[n]$ via the ranks, 
$\pi=\ra_y\circ \ra_x^{-1}$ where $\ra_x$ is given by the first $n$ values
of the $X$-sequence, and $\ra_y$ the analogue for the second coordinate.
We thus obtain a  sequence $(\Pi_n)_{n\in\bN}$ of random permutations, 
with $|\Pi_n|\equiv n$ for all $n\in\bN$. 

By a (two-dimensional) copula we mean the distribution function $C:[0,1]^2\to[0,1]$ associated 
with a probability measure on the unit square that has uniform marginals, i.e.\ $C(u,1)=C(1,u)=u$
for $0\le u\le 1$. With $\mu,F_1,F_2,X_1,Y_1$ as in the previous paragraph an associated copula arises
as the distribution function of the transformed variables,
\begin{equation*}
    C(u,v) = P\bigl(F_1(X_1)\le u, F_2(Y_1)\le v\bigr), \quad 0\le u,v\le 1.
\end{equation*} 
Obviously, $((U_n,V_n))_{n\in\bN}$ with $U_n=F_1(X_n)$, $V_n=F_2(Y_n)$
is a sequence of independent random vectors with distribution function $C$ 
that leads to the same 
sequence of random permutations as the original sequence $(Z_n)_{n\in\bN}$. This 
means that we may use $\mu$ or $C$ as starting point and refer to $(\Pi_n)_{n\in\bN}$ 
as being generated by $\mu$ or by $C$, as we will do repeatedly below.

In the special case that $X_i$ and $Y_i$ are independent we obtain $C(u,v)=uv$, the 
independence copula. At the other extreme, with $X_i=Y_i$, the copula is given by
$C(u,v)=u\wedge v$. Following~\cite{BaGrDelDat}, we write $C=C(G)$ for the delay 
copula associated with the delay distribution with distribution function $G$. Here
$C(G)$ is associated with $(X_i,Y_i)=(X_i,X_i+D_i)$, the $X$-variables are uniformly
distributed on the unit interval, $X_i$ and $D_i$ are independent, and $G$ is the 
distribution function of the delays $D_i$, $i\in\bN$.
 
The above embedding of $\bS$ into $\bR^\bS$ leads to a notion of convergence
for (deterministic) sequences $(\pi_n)_{n\in\bN}$ through the convergence of the pattern 
frequencies $(t(\pi_n,\sigma))_{n\in\bN}$, $\sigma\in\bS$. In their seminal 
work~\cite{Hopp}, Hoppen et al. introduced and discussed this concept
and characterized the limits in the case that $|\pi_n|\to\infty$ as 
two-dimensional copulas, which they called \emph{permutons}: For such a limit $C$,
let $t(C,\sigma)$ be the probability that the random permutation $\Pi_k$ 
in the sequence generated by $C$ is equal to $\sigma\in\bS_k$. Then
$\pi_n\to C$ in the pattern frequency topology means that 
$\lim_{n\to\infty} t(\pi_n,\sigma)=t(C,\sigma)$
for all $\sigma\in\bS$.  Thus, with $C$ we associate 
the function $C^\bS:\bS\to [0,1]$ given by
\begin{equation}\label{eq:defC}
     C^\bS(\sigma)\,  =\,  t(C,\sigma) 
          \, = \,\bP(\Pi_k=\sigma) \quad \text{for all } \sigma\in\bS_k,\,k\in\bN.
\end{equation} 
If $Z_n=(U_n,V_n)$, $n\in\bN$,
are independent random vectors with distribution function $C$ then
\begin{equation}\label{eq:CSdef}
  C^\bS(\sigma) \;
          =\;k!\,\bP\bigl (U_1<\ldots<U_k,
                          V_{\sigma^{-1}(1)}<\ldots<V_{\sigma^{-1}(k)}\bigr).
\end{equation}
From \eqref{eq:back} we obtain the relations
\begin{equation}\label{eq:backC}
  C^\bS(\sigma) =\sum_{\rho\in\bS_l} C^\bS(\rho) \, t(\rho,\sigma)
           \quad\text{for all }k\le l,\,\sigma\in\bS_k.
 \end{equation}
The copula $C$ defines $C^\bS$, but is itself also determined by $C^\bS$, 
see \cite[Lemma 5.1]{Hopp}. Further, for copulas
$C,C_1,C_2,\ldots$ pointwise convergence of the functions $C_n^\bS$ to
$C^\bS$ is equivalent to weak convergence of the associated probability measures, 
see \cite[Lemma 5.3]{Hopp}.

Finally, the connection to $U$-statistics is useful in the present context. 
For a pattern $\sigma\in\bS_k$
let the function $h_\sigma:\bigl([0,1]\times[0,1]\bigr)^{k}\to\{0,1\}$ indicate 
whether or not the permutation 
associated with $(z_1,\ldots,z_k)$ is equal to $\sigma$, i.e., with $1(\,\cdot\,)$ as generic notion
for the indicator function,
\begin{equation}\label{eq:hsigma}
  h_\sigma\big ((u_1,v_1),\ldots,(u_k,v_k)\big )\,
       =\sum_{\tau\in\bS_k}1(u_{\tau(1)}<\ldots<u_{\tau(k)},
                v_{\tau\circ \sigma^{-1}(1)}<\ldots<v_{\tau\circ\sigma^{-1}(k)})
\end{equation}  
for $z_j=(u_j,v_j)\in [0,1]\times [0,1]$, $j=1,\ldots,k$. 
Then
\begin{equation}\label{eq:ustat}
     t(\Pi_n,\sigma)\; = \;
         \frac{1}{\binom{n}{k}}\sum_{1\le j_1<\ldots<j_k\le n}
                       h_\sigma(Z_{j_1},\ldots,Z_{j_k}), \quad n\ge k,
\end{equation}
so that $(t(\Pi_n,\sigma))_{n\ge k}$ is a sequence of $U$-statistics with (symmetric) kernel $h_\sigma$.

\section{A functional central limit theorem}\label{sec:FCLT}
Let $(\Pi_n)_{n\in\bN}$ be a sequence of random permutations generated by a copula $C$.
Then 
\begin{equation}\label{eq:asC}
   t(\Pi_n,\sigma)\,\to\, t(C,\sigma) = C^\bS(\sigma) 
                         \quad\text{ a.s.\ for all } \sigma\in\bS
\end{equation}
as $n\to\infty$; see~\cite{BaGrDelDat} for a proof using $U$-statistics. 
In the `functional view' all patterns are regarded
simultaneously, so that $(T_n)_{n\in\bN}$ with
\begin{equation}\label{eq:defTn}
    T_n(\sigma) = t(\Pi_n,\sigma)\quad \text{for all }\sigma\in\bS,\,n\in\bN,
\end{equation}
becomes a sequence of random variables 
in an infinite-dimensional linear space. We will use a suitably chosen 
Hilbert space $\bH\subset \bR^\bS$.
For this, let $p=(p_\sigma)_{\sigma\in\bS}$ be a probability mass 
function on $\bS$ with the property that $p_\sigma>0$ for all $\sigma\in\bS$, and
let $\bH:=\ell^2(\bS,p)$, with inner product
\begin{equation*}
 \langle f,g\rangle \, =\, \sum_{\sigma\in\bS} p_\sigma f(\sigma)g(\sigma)
            \quad\text{for all }f,g\in\bH,
\end{equation*} 
and associated norm $\| f \|:= \langle f,f\rangle^{1/2}$ for all $f\in\bH$.
Note that $t(\Pi_n,\sigma)^2\le 1$, so that the $T_n$'s in~\eqref{eq:defTn} 
are elements of $\bH$. 
Moreover, it follows from~\eqref{eq:asC} and dominated
convergence that the random functions $T_n$ converge to $C^\bS$ almost surely 
in $\bH$ as $n\to\infty$. Hence, in statistical terms, the pattern frequency functions 
$T_n$ provide a strongly consistent estimator for the function $C^\bS$.
It also follows from~\eqref{eq:defC} and~\eqref{eq:ustat} with $n=k$ that
\begin{equation*}
    ET_n(\sigma) \,=\, ET_k(\sigma) \,=\, C^\bS(\sigma)\quad\text{if } n\ge k:= |\sigma|,
\end{equation*}
meaning that, loosely speaking, $T_n$ is unbiased.

A next step would be the asymptotic normality of a suitably scaled difference between
$\Pi_n$, represented by $T_n$, and the limiting copula $C$, 
represented by the function $C^\bS$.  
Again we consider these as elements of $\bH=\ell^2(\bS,p)$,
which leads to stochastic processes with `time parameter' $\sigma\in\bS$.
We now use the specific mass function $p=(p_\sigma)_{\sigma\in\bS}$ on 
$\bS$ given by 
\begin{equation}\label{eq:defp}
  p_\sigma\,=\, \gamma\, |\sigma|!^{-2}\, 2^{-|\sigma|} \ 
                  \text{ for all }\sigma\in\bS, 
                    \text{ with } \gamma:= e^{-1/2}(1-e^{-1/2})^{-1}.
\end{equation}
Note that $q_k:=\gamma\, k!^{-1}2^{-k}=\sum_{\sigma\in\bS_k}p_\sigma$, $k\in\bN$, 
is the probability mass function  of the zero-truncated Poisson distribution 
with parameter $1/2$. Choosing a
permutation $\sigma\in\bS$ according to $p$ means that first the size 
$|\sigma|=k$ of the permutation is chosen with probability $q_k$, and that the permutation
$\sigma\in \bS_k$ is then chosen uniformly at random, i.e.\ with probability $1/k!$. 
We write `$\toweak$' and `$\todistr$' respectively 
for weak convergence of probability measures 
and distributional convergence
of random variables in $\bH$ as well as in the euclidean spaces~$\bR^d$, $d\in\bN$.
In connection with copulas $C_n\toweak C$ refers to the corresponding probability 
measures on the unit square.

We require an extension of $C^\bS$ to two arguments. With $Z_1,Z_2,Z_2',\ldots$ independent with 
distribution function $C$  let
\begin{equation}\label{eq:defCS2}
  C^{\bS,2}(\sigma,\tau)\, :=\, 
               \bP\bigl(\Pi_k(Z_1,Z_2,\ldots,Z_k)=\sigma, \,\Pi_j(Z_1,Z'_2,\ldots,Z'_j)=\tau\bigr)
\end{equation}
for $\sigma\in\bS_k,\tau\in\bS_j$. Note that $C^{\bS,2}(\sigma,1)=C^{\bS,2}(1,\sigma)=C^\bS(\sigma)$.

In view of later needs we consider two extensions of the usual setup. First, we let the
copula generating $\Pi_n$ depend on $n$. Hence, for each $n\in\bN$, let $C_n$
be a copula and let $\Pi_n$ be the random permutation generated by a sample 
$Z_{n,1},\ldots,Z_{n,n}$ of size $n$ from $C_n$. We will later assume that 
$C_n$ converges weakly to some copula $C$. This  implies the convergence of
$C_n^\bS(\sigma)$ to $C^\bS(\sigma)$ for all $\sigma\in\bS$, see the remarks 
following~\eqref{eq:backC}. We need the following extension of this property 
to the functions defined in \eqref{eq:defCS2}.

\begin{lemma}\label{lem:CS2}
If $C_n\toweak C$ as $n\to\infty$ then 
$\; \lim_{n\to\infty}C_n^{\bS,2}(\sigma,\tau) = C^{\bS,2}(\sigma,\tau)$
for all $\sigma,\tau\in\bS$.
\end{lemma}

\begin{proof} Suppose that $\sigma\in\bS_k$, $\tau\in\bS_j$ where we may assume that $k,j>1$.
From its definition it follows that $C_n^{\bS,2}(\sigma,\tau)$ can be written
as the integral of the function
\begin{equation*}
    (z_1,z_2,\ldots,z_k,z_2',\ldots,z_j')\; \mapsto \; 
                 h_\sigma(z_1,z_2,\ldots,z_k)\, h_\tau(z_1,z_2',\ldots,z_j)
\end{equation*}
with respect to the joint distribution of $Z_1,\ldots,Z_k,Z_2',\ldots,Z_j'$, 
which is the power $k+j-1$ of the measure associated with $C_n$. The latter
converges weakly to the corresponding power of $C$, which 
leads to the integral $C^{\bS,2}(\sigma,\tau)$. The statement thus follows
from the boundedness of the integrand and the fact that the integrand is continuous 
outside the set of arguments with ties in the coordinate values, which is a set 
of probability zero for the limit measure.  
\end{proof}

We now introduce the stochastic process $W_n=(W_n(\sigma))_{\sigma\in\bS}$ of 
scaled differences between $T_n$ and $C^\bS$ defined by 
\begin{equation}\label{eq:defW}
  W_n(\sigma)\, :=\,
        \sqrt{n}\,|\sigma|^{-1}\,\bigl( T_n(\sigma) - C_n^\bS(\sigma)\bigr)\quad\text{if }n \ge k:=|\sigma|, 
\end{equation}
and $W_n(\sigma):=0$ for $n<k$.  In the second of the extensions announced above we introduce
a `truncation parameter' $k=k_n\le n$, leading to processes $W_{n,k_n}$, $n\in\bN$, defined by
\begin{equation}\label{eq:defWtrunc}
  W_{n,k_n}(\sigma)\, :=\,
          \begin{cases}W_n(\sigma), &\text{if }|\sigma|\le k_n,\\ \ 0, &\text{otherwise.}\end{cases}
\end{equation}
Clearly, $W_n=W_{n,n}$. 
Convergence in distribution of the processes as $n\to\infty$ follows from convergence of the associated 
finite-dimensional distributions, see below, and tightness of the sequence. 

Let  $(\mu_n)_{n\in\bN}$ be a sequence of probability distributions $\mu_n$ on some separable infinite-dimensional real 
Hilbert space with inner product $\langle\cdot,\cdot\rangle$ and complete orthonormal basis $\{e_i:\, i\in\bN\}$.
For a probability measure $\mu$ on $\bH$ and $h\in\bH$ let  $\mu^h$ be the push-forward of $\mu$ under 
the map $\bH\ni x\mapsto \langle x,h \rangle\in\bR$. Similarly, for $m\in\bN$ and basis elements $e_1,\ldots,e_m$, 
let $\mu^{e_1,\ldots,e_m}$ be the push-forward of $\mu$ under the map
$\bH\ni x\mapsto (\langle x,e_1\rangle,\ldots,\langle x,e_m\rangle)\in\bR^m$. Suppose now that we have convergence 
of the finite-dimensional distributions in the sense that
\begin{equation}\label{eq:weak_conv2}
    \mu_n^{e_1,\ldots,e_m}\toweak \mu^{e_1,\ldots,e_m}\quad\text{for all}~m\in\bN. 
\end{equation}  
Then tightness will follow if 
 \begin{equation}\label{eq:proho}
 \lim_{k\to\infty}\,\sup_{n\in\bN} \; 
         \sum_{i=k}^\infty \int \langle x,e_i\rangle^2\, \mu_n(dx) \; =\; 0
\end{equation}
is satisfied; see \cite[Section 1.8]{VaartWellner}.

Here we have $\bH=\ell^2(\bS,p)$, in which case a suitable basis
is $\{p_\sigma^{-1/2} 1_\sigma:\, \sigma\in\bS\}$, with $1_\sigma(\tau)=1$ if $\tau=\sigma$ and 0 otherwise. 
Then, with $\mu_n$ as the distribution of $W_{n,k_n}$ on the Hilbert space $\bH$, \eqref{eq:weak_conv2} 
translates to the convergence in distribution of the random vectors $(W_{n,k_n}(\sigma_1),\ldots,W_{n,k_n}(\sigma_m))$
for all $m\in\bN$, $\sigma_1,\ldots,\sigma_m\in\bS$.

\begin{theorem}\label{thm:patternCLT}
Suppose that $(C_n)_{n\in\bN}$ is a sequence of copulas that converges weakly
to some copula $C$ as $n\to\infty$. Suppose further that $(k_n)_{n\in\bN}$ is a sequence 
of nonnegative integers $k_n\le n$ such that $\lim_{n\to\infty}k_n=\infty$. Then,
with $(W_{n,k_n})_{n\in\bN}$ as defined above it holds that
\begin{equation*}
   W_{n,k_n} \; \todistr\, W_C\quad\text{in }\;\bH\;\text{ as }\, n\to\infty,
\end{equation*}
where  $W_C=(W_{C}(\sigma))_{\sigma\in\bS}$ is an 
$\,\bH$-valued centered Gaussian process 
with covariance function $\rho_C:\bS\times\bS\rightarrow \bR$ given by
\begin{equation}\label{eq:cov}
\rho_C(\sigma,\tau)\, :=\, C^{\bS,2}(\sigma,\tau) - C^\bS(\sigma)\,C^\bS(\tau), 
                      \quad\sigma,\tau\in\bS.
\end{equation}
\end{theorem} 

\begin{proof} We consider the finite-dimensional distributions first.

Let $h_\sigma$ be as in \eqref{eq:hsigma}, let $E:=[0,1]^2$ 
and recall that $\Pi_n$ is based on a sample $Z_{n,1},\ldots,Z_{n,n}$ from $C_n$. 
Let $Z_1,\ldots,Z_n$ be a sample from the limit copula $C$.
For $\sigma\in\bS$ with $|\sigma|=:k>1$ and $r=1,\ldots,k$
let $\Hat h_{n;r,\sigma}:E^r\to \bR$, $n\ge |\sigma|$, and   
$\Hat h_{r,\sigma}:E^r\to \bR$ be defined by
\begin{align*}
  \Hat h_{n;r,\sigma}(z_1,\ldots,z_r)\; 
       &=\; Eh_{\sigma}(z_1,\ldots,z_r,Z_{n,r+1},\ldots,Z_{n,k})\; 
                                                   -\; C_n^\bS(\sigma),\\
  \Hat h_{r,\sigma}(z_1,\ldots,z_r)\; &=\; Eh_{\sigma}(z_1,\ldots,z_r,Z_{r+1},\ldots,Z_{k})\; 
       -\; C^\bS(\sigma)
\end{align*}
for all $z=(z_1,\ldots,z_r)\in E^r$. We put 
$\Hat h_{n;r,\sigma}\equiv 0 \equiv \Hat h_{r,\sigma}$ if $|\sigma|=1$. Let 
\begin{equation*}
 \Hat W_n(\sigma):=n^{-1/2}\sum_{i=1}^n \Hat h_{n;1,\sigma}(Z_{n,i}),\quad n\in\bN.
\end{equation*}
Applying the standard arguments for establishing the asymptotic normality of 
$U$-statistics, see e.g.\ \cite[Theorem 12.3]{vdV},  we obtain that 
\begin{equation}\label{eq:extended_projection}
 W_n(\sigma) - \Hat W_n(\sigma) \, \to \, 0
           \quad\text{in probability as }n\to\infty
\end{equation}
for all $\sigma\in\bS$. For $\Hat W_n$ we have the covariance function
\begin{equation*}
\rho_n(\sigma,\tau)\; 
         :=\; \cov\bigl(\hat h_{n;1,\sigma}(Z_{n,1}),\hat h_{n;1,\tau}(Z_{n,1})\bigr) \;
         =\; C_n^{\bS,2}(\sigma,\tau)\,-\, C_n^\bS(\sigma)\, C_n^\bS(\tau).
\end{equation*}
Using Lemma~\ref{lem:CS2} we see that this converges to
\begin{equation}\label{eq:CS2int}
\rho_C(\sigma,\tau)\; 
         :=\; \cov\bigl(\hat h_{1,\sigma}(Z_{1}),\hat h_{1,\tau}(Z_1)\bigr) \;
         =\; C^{\bS,2}(\sigma,\tau)\,-\, C^\bS(\sigma)\, C^\bS(\tau)
\end{equation}
as $n\to\infty$. The multivariate Lindeberg--Feller central limit theorem 
now implies
\begin{equation}\label{eq:Lindeberg_Feller}
 \bigl(\Hat W_n(\sigma_1),\ldots,\hat W_n(\sigma_m)\bigr) 
               \todistr\  \bigl(W_{C}(\sigma_1),\ldots,W_{C}(\sigma_m)\big)
                                \quad\text{as } n\to\infty
\end{equation}
for all $m\in\bN$ and all $\sigma_1,\ldots,\sigma_m\in\bS$; the Lindeberg condition holds 
because $|\hat h_{n;1,\sigma}(Z_{n,i})|\le 1$. In view
of~\eqref{eq:extended_projection} this gives the convergence of the 
finite-dimensional distributions and also shows that 
a centered Gaussian process $W_C=(W_{C}(\sigma))_{\sigma\in\bS}$ 
with covariance function $\rho_C$ exists on $\bR^\bS$.

In order to prove tightness we make use of~\eqref{eq:proho}.   
For $x\in\bR$ and nonnegative integers 
$n$ let
\begin{equation*}
(x)_n:=\begin{cases}1,& n=0,\\
                    x(x-1)\cdots (x-n+1),& n\ge1,
       \end{cases}
\end{equation*}
be the falling factorials. Fix $\sigma\in \bS$. Of course, 
$0=\rho_C(\sigma,\sigma)=E\bigl(W_n^2(\sigma)\bigr)=0$ for all $n\in\bN$ if $\sigma$ 
is the single element of $\bS_1$. If $k=|\sigma|>1$ let
\begin{equation*}
     \zeta_{n,r}(\sigma)\,:=\,E\bigl  (\Hat h_{n;r,\sigma}^2(Z_{n,1},\ldots,Z_{n,r})\bigr)  
                 \quad\text{for }1\le r\le k\le n,
\end{equation*}
and let  $\zeta_{n,r}(\sigma)\equiv 0$ if $k=1$ or $n<k$.
By the well-known formula for the variance of $U$-statistics, see e.g.~\cite[p.\,163]{vdV}, 
for all $n\ge k$,
\begin{equation*}
    E\bigl(W_n^2(\sigma)\bigr)\,
          =\, \frac{1}{k^2}\sum_{r=\max(1,2k-n)}^k
                         \frac{k!^2}{r!(k-r)!^2}
                               \frac{(n-k)_{k-r}}{(n-1)_{k-1}}\,\zeta_{n,r}(\sigma).
\end{equation*}
Because of $|\zeta_{n,r}(\sigma)|\le 1$ this leads to
\begin{equation}\label{eq:schranke1}
  E\bigl  (W_n^2(\sigma)\bigr )\ \le
     \ \frac{1}{k^2} \sum_{r=1}^{k}\frac{k!^2}{r!(k-r)!^2}
     \ \le\ \frac{k!}{k^2}\,\sum_{r=1}^{k}\binom{k}{r}
                  \ \le\ \frac{k!}{k^2}\,2^{k}
\end{equation}
for all $\sigma\in\bS$ with $k=|\sigma|>1$, and all $n\in\bN$. With 
$\gamma$ as in ~\eqref{eq:defp} we get
\begin{equation}\label{eq:basicbound}
 \sum_{\sigma\in\bS_k}E\langle W_n,e_\sigma\rangle^2
    \; =\; \sum_{\sigma\in\bS_k} p_\sigma EW_n(\sigma)^2 
    \;\le \; \frac{\gamma}{k^2}\,.
\end{equation}
Note that the bound does not depend on $n$. This implies
\begin{equation*}
\lim_{k\to\infty}\,\sup_{n\in\bN}\, 
        \sum_{\sigma\in\bS,|\sigma|\ge k} E\bigl\langle W_n,e_\sigma\bigr\rangle^2\; = \; 0,
\end{equation*}
hence Prohorov's criterion is satisfied.  As explained above, it follows from these two steps 
that $W_n\todistr W_C$ as $n\to\infty$.

It remains to consider the effect of truncation. Using~\eqref{eq:basicbound} we get
\begin{equation*}
    E\bigl\| W_n-W_{n,k_n}  \bigr\|^2\ 
               \le \  \sum_{k_n<j\le n}\;\sum_{\sigma\in\bS_j} p_\sigma EW_n(\sigma)^2\
               \le \  \sum_{k_n<j\le n}\frac{\gamma}{j^2} \ =\ o\bigl(k_n^{-1}\bigr),
\end{equation*}
which implies that $W_n-W_{n,k_n}$ converges to 0 in probability as $n\to\infty$.
\end{proof}

We collect some observations on variants and possible extensions.

\begin{remark} (a) We defined $W_n(\sigma)$ to be 0 if $|\sigma|>n$, but the 
theorem remains valid if~\eqref{eq:defW} is used throughout, with the understanding 
that $T_n(\sigma)=0$ in that case. 

(b) The theorem continues to hold if the weight sequence $(p_\sigma)_{\sigma\in\bS}$ 
in~\eqref{eq:defp} is replaced by some other sequence $(r_\sigma)_{\sigma\in\bS}$ 
as long as, with some $c<\infty$, $r_\sigma\le c\,p_\sigma$ for all $\sigma\in\bS$.
In connection with 
variants for other weight sequences we note that the weights are only needed to obtain
tightness. On a technical level, they appear in the definition of the inner product and 
in connection with the bounds in \eqref{eq:schranke1}, \eqref{eq:basicbound}.

(c) On a more qualitative level one might consider different norms, replacing 
for example
\begin{align*}
	\| f\|^2_{q,2} \
	     =\ \sum_{k=1}^\infty q_k\,\frac{1}{k!}\sum_{\sigma\in\bS_k} f(\sigma)^2\quad
\text{by}\quad
   \| f\|_{q,\infty} \ = \ \sum_{k=1}^\infty q_k\sup_{\sigma\in\bS_k} |f(\sigma)|
\end{align*}
for functions $f:\bS\to\bR$.

(d) The processes $W_n$ depend on their parameter $\sigma\in\bS$ through the associated 
kernel $h_\sigma$. Hence, in some analogy to empirical process theory, one might regard these
as $U$-processes with a specific class of kernels, see e.g.\ \cite{delaPenaGine}. However,
in our situation there is no finite uniform upper bound for the order of 
the kernels in the family of interest.   
\end{remark}  


The limit process $W_C$ may be degenerate, for example if $C(u,v)=u\wedge v$. 
However, this will not happen if
\begin{equation}\label{eq:nondeg}
       C^{\bS,2}(\sigma,\sigma)-C^\bS(\sigma)^2\neq 0
             \quad\text{for some }\sigma\in\bS,
\end{equation} 
as the left hand side is the variance of $W_C(\sigma)$. The non-degeneracy 
condition~\eqref{eq:nondeg} is easily seen to hold in the independence case; 
for delay copulas we have the following result. 

\begin{proposition}\label{prop:delay_unique}
Suppose that $C=C(G)$ is the delay copula associated with a distribution function 
$G$.  Then \eqref{eq:nondeg} holds if
$G(0) = 0$,  and $G(\epsilon) > 0$ for all $\epsilon >0$. 
\end{proposition}
\begin{proof}

We first note that, with $C$ the distribution function of $(U,V)$,
$\rho_{C}(12,12)$ is the limiting variance for Kendall's rank correlation coefficient
and may be written as
\begin{equation}\label{eq:var}
    \rho_{C}(12,12)\,=\, 4\,\var\bigl(C(U,V) -(U + V)/2\bigr),
\end{equation} 
see e.g.~\cite{Dengler}. Suppose now that $\rho_{C}(12,12)=0$ for a delay copula $C=C(G)$
with $G$ satisfying the condition given in Proposition~\ref{prop:delay_unique}. Then, for some $c\in\bR$ 
we must have $P(A)=1$ for $A:=\{(u,v):\, C(u,v) - (u+v)/2 - c = 0\}$. Since the function
$(u,v)\mapsto C(u,v) - (u+v)/2 -c$ is continuous, the set $A$ is closed. This implies that
the support $H$ of the measure associated with $C(G)$ is a subset of $A$. 
It is shown in \cite[Section 3]{BaGrDelDat} 
that the set $H_==\{(u,v)\in [0,1]^2: v = F(u)\}$ is a subset of $H$; here 
$F$ is the distribution function of the departure times $X_i+D_i$.
It follows that.
\begin{equation*}
C\bigl(u,F(u)\bigr)\, =\, \bigl(u+F(u)\bigr)/2\,+\,c\quad \text{for all }\,u\in [0,1].
\end{equation*}
Generally, 
\begin{equation*}
  C(G)(u,v)\, = \, \int_0^u G\bigl (F^{-1}(v) - w \bigr )\,dw \quad\text{ for }\,u\in [0,1],~v\in (0,1), 
\end{equation*}
with $F^{-1}(v):=\inf\{y\in \bR: F(y)\ge v\}$. As~$0 = G(0) < G(\epsilon)$~for all $\epsilon >0$ and
$F(u)=\int_0^u G(u - w)\,dw = \int_0^u G(w)\,dw$ for all $u\in [0,1]$, the function $F$ is continuous and
strictly increasing on the unit interval, with $F(0)=0$ and $F(1)<1.$
It follows that $F^{-1}\left (F(u)\right )=u$ and  
\begin{equation*}
  C(G)\bigl (u,F(u)\bigr ))\, = \, \int_0^u G(u - w )\,dw \, = \, F(u) \quad\text{ for }\,u\in (0,1]. 
\end{equation*}
Thus $\bigl (u+F(u)\bigr )/2\,+\,c\,=\,F(u)$, i.e. $F(u)\,=\ u + 2c$ for all $u\in [0,1]$.
Due to $F(0)=0$ it is $c=0$ and therefore $F(u)=u$ for all $u\in [0,1]$, which is in contradiction to
$F(1)<1$. 
\end{proof}

If we replace the condition on $G$ in Proposition \ref{prop:delay_unique} by the stronger condition
$G(0) = 0$, $G(n + \epsilon)  > G(n)$ for all $\epsilon > 0$ and $n\in\bN_0$ with $G(n) < 1$,
then $G$ is determined by $C(G)$, see~\cite[Theorem 4]{BaGrDelDat}. This means that consistent goodness-of-fit or 
two-sample tests 
for delay distributions can then be obtained from the respective tests for the associated
copulas, see the next section.

\section{Nonparametric tests}\label{sec:nonpartests}
Given the (nonparametric) family of two-dimensional copulas three classical testing problems are
considered, the goodness-of-fit (GF), the two-sample (TS), and the symmetry (SY) testing problem.
For GF and SY we assume that a random permutation $\Pi_n$ generated by a sample of
size $n$ taken from some unknown copula $C$ is available, but not the sample itself.
For TS we start with random permutations $\Pi_{m}^{(1)}$ and  $\Pi_{n}^{(2)}$ generated by 
independent samples of size $m$ and $n$ taken from unknown copulas $C_1$ and $C_2$, respectively.
The type of the testing procedures proposed here is similar to that found in numerous papers
dealing with these classical testing problems for a given nonparametric family $\mathscr{M}$ of distributions on $\bR^d$ in the case where the original samples are available; see 
e.g.\ \cite{BaFr2004,BaFr2010,ChenMeintanisZhu,Csorgoe,Gretton,Meintanis,SzRi} and the 
references given therein. Consider, for example, the TS problem with independent samples
$X_1,\ldots,X_m$ and $Y_1,\ldots,Y_n$ taken from distributions $\mu\in\mathscr{M}$
and $\nu\in\mathscr{M}$ respectively. Suppose that the distributions $\zeta\in\mathscr{M}$ are 
uniquely determined by some associated functions $m_\zeta:\bR^d\rightarrow \bR$, 
e.g.\ distribution functions, Fourier transforms, or Laplace transforms. 
The corresponding Cram\'er--von Mises 
(CvM) and Kolmogorov--Smirnov (KS) type distances are given by 
$\int  \bigl (m_\mu(t)-m_\nu(t)\bigr )^2\,\tau(dt)$ and 
$\sup \bigl \{ w(t) \bigl | m_\mu(t) - m_\nu(t) \bigr |: t\in \bR^d\bigr \}$
with appropriate $\sigma$-finite measure $\tau$ and weight function $w \ge 0$ on $\bR^d$.
Then, substituting the unknown functions $m_\mu$ 
and $m_\nu$ by appropriate estimators $m_{X,m}$ and $m_{Y,n}$
based on the first and the second sample, we obtain estimators of these distances that can serve
as test statistics.

Here we similarly use the fact that a copula $C$ is uniquely determined by the function
$C^{\bS}$, and deal with the CvM and KS type distances of copulas $C_1,C_2$ given by 
$\sum_{\sigma\in\bS}\, p_\sigma\,\abs{\sigma}^{-2}\,\bigl | C_1^{\bS}(\sigma) 
     - C_2^{\bS}(\sigma)\bigr |^2$ and   
$\sup_{\sigma\in\bS}\, p_\sigma^{\nf{1}{2}}\,\abs{\sigma}^{-1}\,
       \bigl | C_1^{\bS}(\sigma) - C_2^{\bS}(\sigma)\bigr |$.
Replacing the unknown $C_i ^{\bS}$ by the estimators $t\bigl (\Pi_m^{(i)},\,\cdot \,\bigr )$, $i=1,2$,
we obtain corresponding test statistics for the TS problem. 

A precise description of the respective procedure
will be given in the following subsections. In each case the asymptotic analysis is based on 
Theorem~\ref{thm:patternCLT}. This includes the use of a truncation parameter $k=k(n)$ or $k=k(n,m)$, meaning
that only patterns of length $k$ or less are used; see also Section~\ref{sec:implementation}. On a general level it
is obvious from the rapid growth of the size of $\bS_k$ that little is to be gained from long pattern. For example,
for a data size of $n=100$ there would be about $1.5 \cdot 10^{25}$ patterns of length $k=25$, and the expected 
number of hits for each of these `cells' is about $0.0156$ in the independence case. One may see here a rough
analogy to choosing the bandwidth in density estimation.

\subsection{Cram\'er-von Mises type goodness-of-fit tests}\label{subsec:CvMGF}
Let $(\Pi_n)_{n\in\bN}$ be a sequence of random permutations generated by an unknown
copula $C$. We consider the problem of testing the hypothesis $H_0: C=C_0$ against
the general alternative $H_1: C\neq C_0$ for some given $C_0$ or, equivalently,  
$H_0:\ C^\bS=C^\bS_0 \ \text{ against } H_1:\,C^\bS\neq C^\bS_0$.
This setup includes testing for independence, where $C_0(u,v)=uv$.
We know that the random functions $T_n$ defined in~\eqref{eq:defTn} provide
a consistent estimator for $C^\bS$.
For $k\in\bN$ let $\bS^{\sle}_{ k}:=\bigsqcup_{m=1}^k \bS_m$ be the set of permutations 
with length at most $k$. Defining the process $W_n$ as in \eqref{eq:defW} with $C=C_0$
Theorem \ref{thm:patternCLT} motivates the Cram\'er--von Mises type test statistic
\begin{equation}\label{eq:defCvM}
  {\rm CvM}\gf_{n,k}\,:=\, n \sum_{\sigma\in \bS^{\le}_k} p_\sigma\,|\sigma|^{-2} \left (T_n(\sigma) - C_0^\bS(\sigma) \right )^2\;=\; \| W_{n,k} \|^2,
\end{equation}
truncated at $k\le n$. For a given significance level $\alpha\in (0,1)$ the hypothesis is rejected
if~${\rm CvM}\gf_{n,k}>c_{n,k,\alpha}$, with $c_{n,k,\alpha}$ as upper 
$\alpha$ quantile of ${\rm CvM}\gf_{n,k}$ under $H_0$.
Let 
\begin{equation*}
  {\rm CvM}\gf\,:=\,\|W_{C_0}\|^2\,
           =\,\sum_{\sigma\in \bS} p_\sigma W_{C_0}^2(\sigma),
\end{equation*}
where $W_{C_0}$ is the centered Gaussian process of Theorem \ref{thm:patternCLT}, with $C=C_0$.

\begin{theorem}\label{thm:CvMGF}
Suppose that $C_0$ satisfies the nondegeneracy condition~\eqref{eq:nondeg} and let $k_n\le n$, $n\in\bN$,
be such that $k_n\to\infty$ as $n\to\infty$. 
Then the Cram\'er-von Mises type goodness-of-fit
test asymptotically attains the given level in the sense that
\begin{equation}\label{eq:CvMcorrect}
     \lim_{n\to\infty} P_{C_0}\bigl({\rm CvM}\gf_{n,k_n}\ge c_{n,k_n,\alpha}\bigr)\,=\, \alpha.
\end{equation}
Further, the test is consistent against the general alternative
in the sense that 
\begin{equation}\label{eq:CvMconsistent}
     \lim_{n\to\infty} P_{C}\bigl({\rm CvM}\gf_{n,k_n}\ge c_{n,k_n,\alpha}\bigr)\,=\, 1
           \quad\text{for all } C\neq C_0.
\end{equation}
\end{theorem}

\begin{proof} We first consider the `untruncated' case, with $k_n=n$, omitting the truncation parameter from the
notation. 

Theorem~\ref{thm:patternCLT} with a constant sequence of copulas all equal to 
the hypothetical $C_0$ together with the Continuous Mapping Theorem lead to
${\rm CvM}\gf_n\todistr {\rm CvM}\gf$.  
For~\eqref{eq:CvMcorrect} it remains to show that the distribution function 
associated with the limit is continuous and strictly increasing. For this,
we use the well-known identity 
\begin{equation}\label{eq:chiseries}
  \|W_{C_0}\|^2\,\eqdistr \sum_{\kappa\in \Lambda}\kappa\,L_\kappa
\end{equation}
for the squared norm of the Gaussian random element $W_{C_0}$ in $\bH$, 
where `$\eqdistr$' denotes equality in distribution. In~\eqref{eq:chiseries}
the set $\Lambda$ is the finite or countably infinite set of positive eigenvalues 
of the integral operator $R:\bH \rightarrow \bH$ associated with the
covariance function $\rho_{C_0}$ of $W_{C_0}$, given by
\begin{equation*}
  Rf(\sigma) = \sum_{\tau\in\bS} p_\tau\,\rho_{C_0}(\sigma,\tau)\,f(\tau)
        \quad\text{for all }f\in\bH\,\text{ and }\sigma\in\bS.  
\end{equation*}
Further,  the random variables $L_\kappa$
in \eqref{eq:chiseries} are independent, 
and $L_\kappa$ has the $\chi_{d_\kappa}^2$ distribution with 
$1\le d_\kappa<\infty$ as the multiplicity of the eigenvalue $\kappa$;
see e.g.\ \cite{Vakhania}.
The operator $R$ is of trace class, and it holds that
\begin{equation*}
  E\,{\rm CvM}\gf \,=\, E\,\|W_{C_0}\|^2\,
        =\,\sum_{\sigma\in\bS}p_\sigma\, \rho_{C_0}(\sigma,\sigma)
              \,=\,\sum_{\kappa\in \Lambda}d_\kappa\,\kappa<\infty.
\end{equation*}
By the nondegeneracy assumption, 
$\rho_{C_0}(\sigma,\sigma) > 0$ for some $\sigma\in\bS$,  
hence~$\Lambda$ is non-empty, which implies that
the distribution function of ${\rm CvM}\gf$ is continuous, 0 at 0, and strictly
increasing on the nonnegative half-line.  
From this we deduce that $c_{n,\alpha}\rightarrow c_\alpha$,
with $c_\alpha$ as the the upper $\alpha$ quantile of ${\rm CvM}\gf$, and that
\begin{equation}\label{eq:level}
  \lim_{n\to\infty} P_{C_0}\bigl({\rm CvM}\gf_n > c_{n,\alpha}\bigr) 
                   \;=\; P_{C_0}\bigl({\rm CvM}\gf > c_\alpha\bigr) \;=\; \alpha,
\end{equation}
which is the first part of the theorem if $k_n=n$. For a general truncation sequence we use
the argument at the end of the proof of Theorem~\ref{thm:patternCLT} to obtain that
${\rm CvM}\gf_{n,k_n}-{\rm CvM}\gf_n $ converges to 0 in probability, together with Slutsky's theorem. 

In order to prove~\eqref{eq:CvMconsistent} we fix an
alternative $C\not= C_0$.  It follows from the almost sure convergence in~$\bH$ of $T_n$
to the respective $C^\bS$ that, if $C$ is the true copula,
\begin{equation*}
 \frac{1}{n}\,{\rm CvM}\gf_{n,k_n}\; =\; \sum_{\sigma\in\bS^\le_{k_n} }     
           p_\sigma\,|\sigma|^{-2}\,\bigl(T_n(\sigma)-C_0^\bS(\sigma)\bigr)^2
   \  \to \  \sum_{\sigma\in \bS} p_\sigma\,|\sigma|^{-2}\, \bigl(C^\bS(\sigma)-C_0^\bS(\sigma)\bigr)^2
\end{equation*}
almost surely, where the limit is strictly positive. Thus, under the alternative,
\begin{equation}\label{eq:power}
  \lim_{n\to\infty} P_C\bigl({\rm CvM}\gf_{n,k_n} > c_{n,k_n,\alpha}\bigr) \,=\, 1 ,
\end{equation}
meaning that the test asymptotically rejects the hypothesis with 
probability one.
\end{proof}

The theorem provides an asymptotic justification for the validity and consistency of a particular procedure.
For its implementation a method is needed to obtain the critical values for given $\alpha$, $n$ and~$k$.
In the goodness-of-fit situation we have a simple hypothesis and, with $C_0$ given, a straightforward 
Monte Carlo approximation to the distribution of the test statistic under the hypothesis can be used
to approximate  $c_{n,k_n,\alpha}$ to an arbitrary degree.  

\subsection{Cram\'er-von Mises type two-sample tests}\label{subsec:CvMTS}
In the two-sample problem we start with independent sequences $(\Pi_m^{(1)})_{m\in\bN}$
and $(\Pi_n^{(2)})_{n\in\bN}$ of random permutations generated by unknown copulas $C_1$ and $C_2$ respectively.
We consider testing the hypothesis $H_0: C_1=C_2$  against the
general alternative $H_1:C_1\neq C_ 2$, which we may reformulate as
\begin{equation*}
  H_0:\; C_1^\bS=C_2^\bS \  \text{ against } H_1:\;C_1^\bS\neq C_2^\bS. 
\end{equation*}
In analogy to the goodness-of-fit situation a two-sample test based on $\Pi_m^{(1)}$ and 
$\Pi_n^{(2)}$ suggesting itself is the Cram\'er--von Mises type test based on the test statistic
\begin{equation*}
{\rm CvM}\ts_{m,n,k}\;
          =\;\frac{mn}{m+n} \sum_{\sigma\in \bS_k^{\le} }
                        p_\sigma\,|\sigma|^{-2} \bigl(T_{m}^{(1)}(\sigma) - T_{n}^{(2)}(\sigma)\bigr)^2,
\end{equation*}
where $T_{m}^{(1)}=t(\Pi_{m}^{(1)},\cdot\,)$,
$T_{n}^{(2)}=t(\Pi_{n}^{(2)},\cdot \,)$. Again, $k\le \min(m,n)$ is a truncation parameter.
We define the $\bS$-indexed processes $W_{m}^{(1)}$, $W_{n}^{(2)}$ 
and $W_{m,n}$ by
\begin{equation}\label{eq:T2W}
  W_{m}^{(1)}(\sigma)\,=\, m^{\nf{1}{2}}
       \,|\sigma |^{-1}\,\bigl(T_{m}^{(1)}(\sigma) - C_1^\bS(\sigma)\bigr),\quad
  W_{n}^{(2)}(\sigma)\,=\, n^{\nf{1}{2}}
       \,|\sigma |^{-1}\,\bigl(T_{n}^{(2)}(\sigma) - C_2^\bS(\sigma)\bigr),
\end{equation}
and
\begin{equation*}
  W_{m,n}(\sigma)\,=\,\Bigl(\frac{mn}{m+n}\Bigr)^{\nf{1}{2}}
       \,|\sigma|^{-1}\,\bigl(T_{m}^{(1)}(\sigma) - T_{n}^{(2)}(\sigma)\bigr) 
\end{equation*}
for all $\sigma\in\bS$. Then the test statistic can be written as
\begin{equation*}
  {\rm CvM}\ts_{m,n,k}\,=\,\sum_{\sigma\in\bS_k^{\le}} p_\sigma W_{m,n}^2(\sigma).
\end{equation*}
Further, if  $C_1=C_2$, 
\begin{align*}
  W_{m,n}\,=\,\Bigl(\frac{n}{m+n}\Bigr)^{\nf{1}{2}} W_{m}^{(1)}\,
                -\,\Bigl(\frac{m}{m+n}\Bigr)^{\nf{1}{2}} W_{n}^{(2)}.
\end{align*}
In what follows, we assume that the sample sizes $m,n$ tend to infinity such 
that the limit
\begin{equation}\label{eq:limsize}
  \lim_{m,n\to\infty}\,\frac{m}{m+n}\,=:\,\gamma\in (0,1) 
\end{equation}
exists. Let $W_{C_1}^{(1)}$, $W_{C_2}^{(2)}$ be independent copies of the Gaussian 
process $W_C$ in Theorem~\ref{thm:patternCLT}, with $C=C_1$ and $C=C_2$ respectively,
so that $W_{m}^{(1)} \todistr W_{C_1}^{(1)}$ and $W_{n}^{(2)} \todistr W_{C_2}^{(2)}$
as $m,n\to\infty$. Due to independence we even have joint convergence
\begin{equation*}
   \bigl(W_{m}^{(1)}, W_{n}^{(2)}\bigr)\; 
                 \todistr\;  \bigl(W_{C_1}^{(1)},W_{C_2}^{(2)}\bigr)
                      \quad\text{as }m,n\to\infty
\end{equation*}
in the product space $\bH\times\bH$.
In particular, if $C_1=C_2=:C$, and as $m,n\to\infty$,
\begin{equation*}
   W_{m,n}\; \todistr\, (1-\gamma)^{\nf{1}{2}}\,W_{C}^{(1)}-\gamma^{\nf{1}{2}}\,W_{C}^{(2)}
       \;\eqdistr\; W_{C}
 \end{equation*}
and we obtain  ${\rm CvM}\ts_{n,m,k(n,m)}\todistr {\rm CvM}\ts:=\|W_{C}\|^2$  as in
Section~\ref{subsec:CvMGF}, provided that $k(n,m)\to\infty$. 

For a test we need critical values. In contrast to the situation in Section~\ref{subsec:CvMGF} we no longer have a simple 
hypothesis, i.e.\ we do not know the shared copula $C=C_1=C_2$ of the samples under $H_0$.
We suggest a variant of
the bootstrap procedure. 

Given the available data, i.e.~the two separate permutations
$\Pi^{(1)}_{m}$ and  $\Pi^{(2)}_{n}$, we need a suitable notion of resampling. 
Evidently, the permutations
cannot be used directly. Instead, following \cite[Definition 3.4]{Hopp} we define, for 
$\sigma\in \bS$, the copula $\check{C}(\sigma)$ by its Lebesgue density
\begin{equation*}
    f_\sigma(u,v)\,=\,n \cdot 1\bigl(\sigma (\lceil n u \rceil) = \lceil n v \rceil\bigr),
                          \quad (u,v) \in [0,1]^2, 
\end{equation*}
and use 
\begin{equation*}
  \check{C}_{m,n}\,:=\,\frac{m}{m+n}\,\check{C}(\Pi_m^{(1)})\, + \, \frac{n}{m+n}\,\check{C}(\Pi_n^{(2)}) 
\end{equation*}
for constructing the resamples. Conditionally on $\Pi_{m}^{(1)}$, $\Pi_{n}^{(2)}$,
let $\Pi_{m,n}^{*(1)}$, $\Pi_{m,n}^{*(2)}$ be the random permutations based
on independent samples of size $m,n$ from  $\check{C}_{m,n}$. Further, with 
$T_{m,n}^{*(1)}(\,\cdot\,)\,=\,t\bigr (\Pi_{m,n}^{*(1)},\,\cdot\,\bigl)$ and
$T_{m,n}^{*(2)}(\,\cdot\,)\,=\,t\bigr (\Pi_{m,n}^{*(2)},\,\cdot\,\bigl)$ let
\begin{equation}\label{eq:boot}
      W_{m,n}^*(\sigma)\,
         =\,\left (\frac{mn}{m+n}\right )^{\nf{1}{2}}\,|\sigma |^{-1}\, 
           \bigl(T_{m,n}^{*(1)}(\sigma) - T_{n,n}^{*(2)}(\sigma)\bigr)
\end{equation} 
for all $\sigma\in\bS$. We now view
\begin{equation*}
  {\rm CvM}_{m,n,k}\sts\,=\,\sum_{\sigma\in \bS_k^{\le }}p_\sigma\,W_{m,n}^{*^2}(\sigma)
\end{equation*}
as the bootstrap version of the test statistic ${\rm CvM}\ts_{m,n,k}$. For a given level 
$\alpha\in (0,1)$ let
\begin{align*}
  c_{m,n,k;\alpha}^*\big (\Pi_{m}^{(1)},\Pi_{n}^{(2)}\big ) 
           &:=\inf\Bigl \{t\in\bR: P\bigl(\,{\rm CvM}_{m,n,k}\sts > t\,
                  \big |\,\Pi_{m}^{(1)},\Pi_{n}^{(2)}\,\big )\le \alpha\Bigr\} ,
\end{align*}                          
where the conditional probability on the right hand side can be approximated to arbitrary
precision by a sufficiently large number of bootstrap resamples. In the following result 
we consider the asymptotic behavior of the two-sample Cram\'er--von Mises type test that 
rejects the hypothesis if the test statistic exceeds the bootstrap critical value. 

\begin{theorem}\label{thm:CvMTS} Suppose that $C_1$ and $C_2$ both satisfy 
the nondegeneracy condition~\eqref{eq:nondeg} and that~\eqref{eq:limsize} 
holds for the sample sizes. Let the truncation values $k(n,m)\in\bN$ be such that 
$k(n,m)\le \min\{n,m\}$ and $\lim_{n,m\to\infty}k(m,n)=\infty$.  Then 
\begin{equation*}
    \lim_{m,n\to\infty} 
     P_{C_1,C_2}\left ({\rm CvM}_{m,n,k(n,m)}\ts>c_{m,n,k(n,m);\alpha}^*
      \big (\Pi_{m}^{(1)},\Pi_{n}^{(2)}\big ) \right )\;
     =\; \begin{cases}\alpha, &\text{if } C_1 = C_2,\\
                      1, &\text{if } C_1 \neq C_2.     
                    \end{cases}
\end{equation*}
\end{theorem}

\begin{proof}  We may assume that $k(n,m)=\min\{n,m\}$ and then omit the truncation parameter from the notation.
Let
\begin{equation*} 
  C_{\gamma,C_1,C_2}\,:=\,\gamma\,C_1\, + \,(1-\gamma)\,C_2. 
\end{equation*}
From \cite[Lemma 3.5]{Hopp} we obtain
that $\check{C}(\sigma_n)\toweak C$ if $\sigma_n$ converges to $C$ in the pattern frequency 
topology, hence it holds that $\check{C}_{m,n}$ converges weakly to $C_{\gamma,C_1,C_2}$ 
with probability one as $m,n\to\infty$. Let $W_{m,n}^{*(1)}$, $W_{m,n}^{*(2)}$ be defined in terms of $T_m^{*(1)}$, $T_{n}^{*(2)}$
as in~\eqref{eq:T2W}, i.e.,
\begin{align*}
  W_{m,n}^{*(1)}\,=\,m^{\nf{1}{2}} \left (T_{m,n}^{*(1)} - \check C_{m,n}^{\bS}\right ),
                  \quad W_{m,n}^{*(2)}\,=\,n^{\nf{1}{2}} \left (T_{m,n}^{*(2)} - \check C_{m,n}^{\bS}\right ).
\end{align*}
Notice that
\begin{equation}\label{eq:mix}
  W_{m,n}^{*}\,=\,\left (\frac{n}{m+n}\right )^{\nf{1}{2}} W_{m,n}^{*(1)}\, - \,\left (\frac{m}{m+n}\right )^{\nf{1}{2}} W_{m,n}^{*(2)}.
\end{equation}
Then Theorem~\ref{thm:patternCLT}, now with a weakly convergent
sequence of base copulas, implies that almost surely conditionally on the sequences
$\Pi^{(1)}_{m}$, $\Pi^{(2)}_{n}$,
\begin{equation*}
   \bigl(W_{m,n}^{*(1)}, W_{m,n}^{*(2)}\bigr)\; 
                 \todistr\;  \bigl(W_{C_{\gamma,C_1,C_2}}^{(1)},W_{C_{\gamma,C_1,C_2}}^{(2)}\bigr)
\end{equation*}
and, using~\eqref{eq:limsize} and \eqref{eq:mix}, 
\begin{equation*}
  W_{m,n}^*\; \todistr\, (1-\gamma)^{\nf{1}{2}}\, W_{C_{\gamma,C_1,C_2}}^{(1)}\,-\,\gamma^{\nf{1}{2}}\, W_{C_{\gamma,C_1,C_2}}^{(2)} \eqdistr\; W_{C_{\gamma,C_1,C_2}}
\end{equation*}
as $m,n\to\infty$. Thus, almost surely conditionally on the sequences
$\Pi^{(1)}_{m}$, $\Pi^{(2)}_{n}$,
\begin{align*}
{\rm CvM}_{m,n}\sts \;\todistr\, \|W_{C_{\gamma,C_1,C_2}}\|^2\quad\text{ as }m,n\to\infty.
\end{align*}  
From this we deduce that almost surely
 \begin{equation*}
   \left (\frac{mn}{m+n}\right )^{-1} c_{m,n;\alpha}^*\big (\Pi_{m}^{(1)},\Pi_{n}^{(2)}\big ) \rightarrow 0\quad\text{ as }m,n\to \infty.
\end{equation*}
Moreover, in the hypothesis case, with $C_1=C_2=C$, it holds that $W_{C_{\gamma,C_1,C_2}}= W_{C}$ and we
then also have that, almost surely conditionally on the sequences
$\Pi^{(1)}_{m}$, $\Pi^{(2)}_{n}$,
\begin{equation*}
{\rm CvM}_{m,n}\sts \; \todistr\, \|W_{C}\|^2 = {\rm CvM}\ts\quad\text{ as }m,n\to \infty.
\end{equation*}
Denoting by $c_{\alpha}$ the upper $\alpha$ quantile of ${\rm CvM}\ts$, we obtain 
$c_{m,n;\alpha}^*\big (\Pi_{m}^{(1)},\Pi_{n}^{(2)}\big )\rightarrow c_{\alpha}$ almost surely.  
Now the statements of the theorem follow as in the case of Theorem~\ref{thm:CvMGF}.
\end{proof}

\subsection{Cram\'er-von Mises type tests for symmetry}\label{subsec:CvMSY}
If the random vector $Z=(U,V)$ is such that $U$ and $V$ are uniformly
distributed on the unit interval  then the same
holds for $\bar Z:=(V,U)$, and the associated copulas are obviously related 
by $\bar C(x,y)=C(x,y)$, $0\le x,y\le 1$. We say that $C$ 
is \emph{symmetric} if $\bar C=C$. An obvious example is
the case where $U$ and~$V$ are independent.
Important families with this property are  
Archimedean copulas and the Farlie-Gumbel-Morgenstern copulas, see also 
Section~\ref{subsec:FGMtest} below. Nonparametric test based on empirical copulas for
this and other notions of symmetry are discussed in~\cite{JASAsymm}.

As in subsection \ref{subsec:CvMGF}, let $(\Pi_n)_{n\in\bN}$ be a sequence of random permutations generated by an unknown
copula $C$. In order to introduce a
pattern-based test for the symmetry hypothesis $H_0: \, C=\bar C$ we use the
elementary observation that, for all $n\in\bN$ and $\sigma\in\bS_n$, 
\begin{equation}\label{eq:conjugate}
   \Pi_n(Z_1,\dots,Z_n) = \sigma\  \Longleftrightarrow\ 
              \Pi_n(\bar Z_1,\dots,\bar Z_n) = \sigma^{-1}.
\end{equation}
In particular, $C^\bS(\sigma)=\bar C^\bS(\sigma^{-1})$, and  
Theorem~\ref{thm:patternCLT} implies that, with
$\bar T_n(\sigma):= T_n(\sigma^{-1})$, $\bar W_n(\sigma):= W_n(\sigma^{-1})$, and $\bar W_C(\sigma):=W_C(\sigma^{-1})$,
\begin{equation*}
  \bigl(W_n,\bar W_n\bigr)      \; \todistr\, (W_C,\bar W_C)
\end{equation*}
as $n\to\infty$. Hence, the processes  $W\sy_n:=W_n - \bar W_n$ converge in
distribution to $W_C\sy$ with $W\sy_C(\sigma):=W_C(\sigma)-W_C(\sigma^{-1})$
for all $\sigma\in\bS$.
The limit is again a centered Gaussian process, and bilinearity leads
to an expression of its covariance function in terms of the covariances of 
$W_C$ given in~\eqref{eq:cov}.

We may now proceed as in the previous two sections,  taking
\begin{equation*}
  {\rm CvM}\sy_{n}:=n\sum_{\sigma\in\bS_n^{\le}}p_\sigma\,|\sigma |^{-2} \left (T_n(\sigma) - T_n(\sigma^{-1})\right )^2
\end{equation*}
as test statistic. In the hypothesis case, $C^\bS = \bar C^\bS$,  ${\rm CvM}\sy_n = \| W\sy_n \|^2$, 
\begin{equation*}
  {\rm CvM}\sy_{n,k}\; \todistr\, \|W_C\sy\|^2,
\end{equation*}
and we may use the bootstrap to obtain critical values. We leave it to the reader to formulate 
a symmetry analogue of Theorem~\ref{thm:CvMTS} together with a truncation extension.

Finally, we mention that there are other notions of symmetry, for example
those based on $(U,V)\mapsto (U,1-V)$ or $(U,V)\mapsto (1-U,1-V)$ instead of
$(U,V)\mapsto (V,U)$; see~\cite{JASAsymm}. 

\subsection{Kolmogorov-Smirnov type tests}\label{subsec:KS}
The Cram\'er-von Mises tests in the preceding sections are closely connected 
to $\ell^2$-spaces. Here we  briefly consider Kolmogorov-Smirnov type tests, which 
are based on suprema and are therefore similarly related to $\ell^\infty$-spaces. 
We restrict ourselves to a rough outline of the untruncated case.

In contrast to $\ell^2(\bS)$ the space $\ell^\infty(\bS)$ of all bounded real functions on $\bS$ 
is not separable. The Hoffmann-J{\o}rgensen extension of distributional convergence 
makes it possible to work with the full $\ell^\infty$-space, see e.g.~\cite{VaartWellner}. 
Alternatively, as pointed out by a referee, we may work with the closed separable subspace
$c_0(\bS)$ that consists of all functions $x:\bS\to \bR$ that vanish at infinity in the 
sense that for all $\epsilon>0$ there exists a $k=k(\epsilon)<\infty$ such that 
$|x_\sigma|\le \epsilon$ for all $\sigma\in\bS$ with $|\sigma|\ge k$.
A central limit theorem for these  spaces can be obtained from Theorem~\ref{thm:patternCLT}:
Let $\Psi:\bH\to\ell^\infty(\bS)$ be defined by  $\Psi(x)=(p_\sigma^{\nf{1}{2}} x_\sigma)_{\sigma\in\bS}$ for all
$x=(x_\sigma)_{\sigma\in\bS}\in\bH$. This is easily seen to be a  bounded linear operator and, in fact, its range is contained in $c_0(\bS)$. Hence convergence of $\Psi(W_n)$, i.e.\ of $W_n$ in
$\ell^\infty(\bS)$ or $c_0(\bS)$,
follows with a suitable version of the continuous mapping theorem.

Let $\| \cdot \|_\infty$ be the supremum norm on $\ell^\infty(\bS)$. The Kolmogorov--Smirnov type goodness-of-fit test statistic is
\begin{equation*}
  {\rm KS}\gf_n\,=\ \sqrt{n} \max_{\sigma \in \bS^{\le}_ n} p_{\sigma}^{\nf{1}{2}}\,| \sigma |^{-1} \, | T_n(\sigma) - C_0^{\bS}(\sigma) |
  \, =\, \sup_{\sigma\in \bS}\, p_\sigma^{\nf{1}{2}}\,\bigl |W_n(\sigma)\bigr|
  \, =\, \| p_\cdot^{\nf{1}{2}}\, W_n(\cdot) \|_\infty,
\end{equation*}
with distributional limit
\begin{equation*}
         {\rm KS}\gf\,:=\,\|p_\cdot^{\nf{1}{2}}\,W_{C_0}(\,\cdot\,)\|_\infty\,
           =\,\sup_{\sigma\in \bS}\, p_\sigma^{\nf{1}{2}}\bigl|W_{C_0}(\sigma)\bigr|.
\end{equation*}
With $q_{n,\alpha}$ as upper $\alpha$ quantile of ${\rm KS}\gf_n$ when $C = C_0$ obtains,          
the hypothesis is rejected if ${\rm KS}\gf_n > q_{n,\alpha}$. 
This leads to the following Kolmogorov-Smirnov analogue of Theorem~\ref{thm:CvMTS}. For the 
asymptotics of the $\ell^2$-based tests certain properties of the limit distribution
of the test statistic under the hypotheses were important. 
The distribution function of $\,\|p_\cdot^{\nf{1}{2}}\,W_{C_0}(\,\cdot\,)\|_\infty$
has these properties as well; see \cite{BeranMillar} and \cite{Tsirelson}.

\begin{theorem}\label{thm:KS}
Suppose that $C_0$ satisfies condition~\eqref{eq:nondeg}. 
Then the Kolmogorov-Smirnov type goodness-of-fit
test asymptotically attains the given level in the sense that
\begin{equation}\label{eq:KScorrect}
     \lim_{n\to\infty} P_{C_0}\bigl({\rm KS}_n\gf\ge q_{n,\alpha}\bigr)\,=\, \alpha,
\end{equation}
and the test is consistent against the general alternative
in the sense that 
\begin{equation}\label{eq:KSconsistent}
     \lim_{n\to\infty} P_{C}\bigl({\rm KS}_n\gf\ge q_{n,\alpha}\bigr)\,=\, 1
           \quad\text{for all } C\neq C_0.
\end{equation}
\end{theorem}

Similarly we obtain a Kolmogorov-Smirnov type test for the
two-sample problem, together with an analogue of Theorem~\ref{thm:CvMTS}, and 
the bootstrap can again be used to obtain critical values. The same holds for
the symmetry situation.

\section{Tests for parametric families}\label{sec:parametric}
Let $\mathscr{C} = \{C_\theta:\, \theta\in\Theta\}$ be a parametric
family of two-dimensional copulas. Suppose we observe $\Pi_n$ for some $n\in\bN$,
where $\Pi_n$ is generated from $C_\theta$ with an unknown $\theta\in\Theta$.
For a given subset $\Theta_0$ of $\Theta$ we consider tests of the hypothesis
$\theta\in\Theta_0$  
that are based on linear pattern statistics. 
In Subsection \ref{subsec:FGMtest} we deal with
different tests of this type for the  family $\mathscr{C}$ of
Farlie-Gumbel-Morgenstern copulas and compare these  by their asymptotic slopes, which
determine the asymptotic power performance of the tests under local alternatives.
In Subsection \ref{subsec:delaytest},
$\mathscr{C}$ is the family of delay copulas with exponential delay distributions.
We treat a one-sided test for the parameter of the exponential
distribution that is based on the relative frequency~$I_n=t(\Pi_n,21)$ of inversions in $\Pi_n$. 
We compare the asymptotic performance of the test with
that of the uniformly most powerful test for the situation where observations of the delays
$D_1,\ldots,D_n$ are available, using the asymptotic relative efficiency concepts of
Pitman and Bahadur.

\subsection{Farlie-Gumbel-Morgenstern copulas}\label{subsec:FGMtest}
Let $\mathscr{C}$ be the family of Farlie-Gumbel-Morgenstern (FGM) copulas
\begin{equation}\label{eq:defFGM}
    C_\theta(u,v) = C_\theta^{\rm FGM}(u,v)\; =\; uv\, +\, \theta u v (1-u)(1-v), \quad 0\le u,v\le 1,
  \end{equation}
with parameter $\theta\in[-1,1]$; see~\cite{Farlie,Gumbel,Morgenstern}.
In the case $\theta=0$ the components of a two-dimensional
random vector with this distribution function are independent, they are 
positive quadrant dependent if $\theta >0$ and negative quadrant dependent
if $\theta<0$. A two-sided test of the hypothesis $H_0:\theta=0$ against the 
alternative $H_1: \theta\not=0$ appears as a test for independence under
the parametric assumption that the true distribution is from the FGM family.
The tests of $H_0:\theta= 0$ against $H_+: \theta > 0$ or $H_-: \theta < 0\,$ 
similarly apply in the context of quadrant dependency.

If $(U,V)\sim C_\theta$, meaning that the random vector $(U,V)$ has distribution 
function $C_\theta$,  then obviously  $(V,U)\sim C_\theta$, as already mentioned
in Section~\ref{subsec:CvMSY}. We then also have $(1-U,V)\sim C_{-\theta}$, 
$(U,1-V)\sim C_{-\theta}$,  and $(1-U,1-V)\sim C_\theta$.

Let $\pi\in\bS_n$ be the value of a random permutation $\Pi_n$ generated by  
$C_\theta$. Kendall's rank correlation test (Kendall's tau) is based on the 
test statistic $K_2(\Pi_n)$ with $K_2(\pi)=t(\pi,12)-t(\pi,21)$.
The analogue with concordant and discordant triples would suggest
the test statistic $K_3(\Pi_n)$ with $K_3(\pi):= t(\pi,123)-t(\pi,321)$.
Let  $\sigma_i,\,i=1,\ldots,6$, be the elements of $\bS_3$ in
lexicographic order. We consider  
tests that are based on general linear pattern statistics 
$L_a(\Pi_n):= \sum_{i=1}^6a_i\,t(\Pi_n,\sigma_i)$, with $a=(a_1,\ldots,a_6)\tr\in\bR^6$. 
For testing $H_0:\theta= 0$ against $H_+: \theta > 0$ the $L_a$ test 
at a given level $\alpha\in (0,1)$ rejects
the hypothesis if $L_a(\Pi_n)>c_{a;\alpha,n}$, where $c_{a;\alpha,n}$ denotes the upper $\alpha$
quantile of the distribution of $L_a(\Pi_n)$ in the case where the hypothesis is true.
A natural question is whether there are $L_a$ tests that perform 
better than Kendall's rank correlation test. 

We need some values of the functions $C^\bS_\theta$.
From~\eqref{eq:CSdef} we obtain for $\sigma\in\bS_3$
\begin{equation} \label{eq:Cfunc1FGM}
\begin{aligned}
   &C^\bS_\theta(123) = \frac{1}{6}+\frac{\theta}{12}+\frac{\theta^2}{100},
   \quad C^\bS_\theta(132) = \frac{1}{6}+\frac{\theta}{24}-\frac{\theta^2}{200},
   \quad C^\bS_\theta(213) = \frac{1}{6}+\frac{\theta}{24}-\frac{\theta^2}{200},\\
   &C^\bS_\theta(231) = \frac{1}{6}-\frac{\theta}{24}-\frac{\theta^2}{200},
   \quad C^\bS_\theta(312) = \frac{1}{6}-\frac{\theta}{24}-\frac{\theta^2}{200},
   \quad C^\bS_\theta(321) = \frac{1}{6}-\frac{\theta}{12}+\frac{\theta^2}{100}.
\end{aligned}
\end{equation}
Note that the symmetries of $\mathscr{C}$  lead to $C^\bS_\theta(ijk)=C^\bS_{-\theta}(kj i)$
for all $ijk\in\bS_3$.
For a linear pattern statistic,
$E_\theta L_a(\Pi_n)= \sum_{i=1}^6a_i\,C^\bS_\theta(\sigma_i)=:\mu_a(\theta)$,
and $E_\theta K_3(\Pi_n)=\theta/6$ is a special case.

For an asymptotic comparison we require the respective limit variances
under the hypothesis, and in order to be able to apply Theorem~\ref{thm:patternCLT} certain  
values of the function $C_0^{\bS,2}$ are needed. We can use the approach developed in~\cite{JNZ}
for the independence copula. 
In particular,  
\cite[Example 2.6]{JNZ} provides the asymptotic
covariance matrix $\Xi$ of the six-dimensional random vector 
$\sqrt{n}\bigl (t(\Pi_n,\sigma_i)-C_0^\bS(\sigma_i)\bigr)_{i=1}^6$ as
\begin{equation}\label{eq:defSigma0}
       \Xi\, := \, \frac{1}{400}\, \begin{pmatrix}
                               26 &12 &12 &-13 &-13 &-24 \\
                               12 &14 &-1 &-6 &-6 & -13  \\
                               12 &-1 &14 &-6 &-6 & -13  \\
                              -13 &-6 &-6 &14 &-1 & 12  \\
                              -13 &-6 &-6 &-1 &14 & 12   \\
                              -24 &-13 &-13 &12 &12 & 26 \end{pmatrix}.
\end{equation}
(In~\cite{JNZ} the authors use the norming $n^{|\sigma|}$ 
which accounts for a factor 36.)  For example, 
this implies that, under the hypothesis, $\sqrt{n}K_3(\Pi_n)$ is asymptotically normal with limiting
variance $1/4$.

To deal with the asymptotic performance of the $L_a$ test, fix $h\in [-1,1]$, and 
let the permutation $\Pi_n$ be generated from the copula $\tilde C_{h,n}:=C_{h/\sqrt{n}}\in \mathscr C$.
Clearly, $\tilde C_{h,n}$ converges to $C_0$ as $n\to\infty$. Theorem~\ref{thm:patternCLT} implies  
\begin{equation*}
\sqrt{n}\bigl(t(\Pi_n,\sigma_i)-\tilde C_{h,n}^\bS(\sigma_i)\bigr)_{i=1}^6
       \ \todistr\ N_6(0,\Xi).
\end{equation*}
The expressions obtained above for $C_\theta(\sigma_i)$ lead to
\begin{equation}\label{eq:defm}
\sqrt{n}\bigl(\tilde C_{h,n}^\bS(\sigma_i)- C_0^\bS(\sigma_i)\bigr)_{i=1}^6
       \; \rightarrow\; h\cdot m,\quad\text{with }
                  m:=\frac{1}{24}\, (2,1,1,-1,-1,-2)\tr.
\end{equation}
Consequently,
\begin{equation*}
\sqrt{n}\bigl(t(\Pi_n,\sigma_i)-C_0^\bS(\sigma_i)\bigr)_{i=1}^6
\ \todistr\ N_6(hm,\Xi).
\end{equation*}
Let $\mu_a:=\mu_a(0)$. By the above limit considerations, $L_a(\Pi_n)$ is asymptotically  
normal in the sense that
\begin{equation}\label{eq:asvar0}
    \sqrt{n}\bigl(L_a(\Pi_n)-\mu_a\bigr)\ \todistr\  N(h \,a\tr m, a\tr \Xi a).
\end{equation}
Suppose that $v(a):=a\tr \Xi a>0$ and that $a\tr m>0$. Let $z_\alpha$ be the upper $\alpha$ quantile of the standard
normal distribution. Then if the hypothesis is true
\begin{equation*}
\sqrt{n}\bigl(L_a(\Pi_n)-\mu_a\bigr)\ \todistr N\left (0,v(a)\right ),\quad 
c_{a;\alpha,n}=\mu_a + z_\alpha\sqrt{v(a)/n}+o(1/\sqrt{n}),
\end{equation*}
and the $L_a$ test is of asymptotic level
$\alpha$. As explained e.g.\ in~\cite[Chapter 14]{vdV} 
this test has \emph{asymptotic slope}
in $\theta=0$ 
given by $s(a):= a\tr m/\sqrt{v(a)}$ in the sense that, for $h>0$,
\begin{equation}\label{eq:aslocpower}
      \lim_{n\to\infty}\phi_n(h\, n^{-1/2}) \; = 
            \; 1- \Phi\bigl(z_\alpha-h \cdot s(a)\bigr).
\end{equation}
Here $\phi_n$ is the power function of the test and $\Phi$ denotes the standard normal
distribution function. For example, with
$a=(1,0,0,0,0,-1)\tr$, which corresponds to the above $K_3(\Pi_n)$, we obtain  $s(a)=1/3$.

What is the optimal linear combination if all patterns of length three are known?
With maximal slope as a criterion we have to 
maximize $s(a)$ on the set of vectors $a\in\bR^6$ with $v(a)>0$ and $a\tr m>0$.
The matrix $\Xi$ is positive semidefinite and has rank 4. 
Let $e_1,\ldots,e_6$ be an orthogonal set of eigenvectors 
for the eigenvalues  $\lambda_1>\lambda_2 >\lambda_3>\lambda_4>\lambda_5=\lambda_6=0$
of $\Xi$. Then, with $\eta_i=\eta_i(a):=a\tr e_i/\|e_i\|$,
\begin{equation}\label{eq:sl2}
    s(a)^2 = \frac{(a\tr m)^2}
              {\phantom{\big|}
                \lambda_1\eta_1^2 +\lambda_2\eta_2^2
              +\lambda_3 \eta_3^2 + \lambda_4 \eta_4^2},
\end{equation}
The eigenvalues and a specific set of orthogonal eigenvectors have been given 
in~\cite{JNZ}. With $e_1=(2,1,1,-1,-1,-2)\tr = 24\cdot m$, $e_5=(1,-1,-1,1,1,-1)\tr$ and $e_6=(1,1,1,1,1,1)\tr$ 
it is then clear that $a\tr m>0$ implies $v(a)>0$, and that  the set of vectors $a\in\bR^6$
that maximizes $s^2$ is equal to    
\begin{equation}\label{eq:Sopt}
 S(m,\Xi) := \bigl\{a=\eta_1 e_1+\eta_5e_5+\eta_6 e_6:\,
                                \eta_1 > 0,\eta_5,\eta_6\in\bR\}.
\end{equation}
Further, in view of $\lambda_1=3/16$, the maximum is equal to $1/9$, which is the square of the slope for
$K_3(\Pi_n)$ derived above. Indeed, with the corresponding $a = (1,0,0,0,0,-1)\tr=:a^*$
we get  
$a^* = \frac{1}{3}\, e_1 + \frac{1}{3}\, e_5\, \in\, S(m,\Xi)$,
which shows that the above heuristically motivated test statistic is optimal in
the sense discussed here. If we only count the ascending triples then
$a = (1,0,0,0,0,0)\tr =:a^\uparrow$ and $s(a^\uparrow)^2=\frac{25}{26\cdot 9}=\frac{25}{26}\hspace*{0.5mm} s(a^*)^2$ for the resulting 
test statistic  $L_{a^\uparrow}(\Pi_n)$. 

It can be shown that for tests $L_a$ and $L_b$ with  $a\tr m>0$ and $b\tr m >0$
the quotient $s(b)^2/s(a)^2$ is the limit of the ratio of the sample sizes
needed for the $L_a$ test and the $L_b$ test to achieve a given power, and thus is      
the Pitman asymptotic relative efficiency (ARE) of the $L_b$ test with respect to the $L_a$ test;
see~\cite[Chapter 14]{vdV} or~\cite[Section 13.2]{leh} for a thorough treatment of this concept,
and also Subsection \ref{subsec:delaytest} for a definition and derivation of this and other
asymptotic relative efficiencies in the special situation treated there.
 
From this point of view about $8\%$ 
more observations are needed in comparison to $K_3(\Pi_n)$ if we only count 
the ascending triplets. We summarize the above in the following result.
An analogous result for testing $\theta=0$ against $\theta<0$ is readily obtained 
on using the relation between $ C_{-\theta}$ and $ C_\theta$ given at the
beginning of this subsection. 

\begin{proposition}\label{prop:FGM3}
In  the  FGM family of copulas consider the tests based on $L_a,L_b$ for $\theta=0$ 
against $\theta>0$, with $a\tr m>0$, $b\tr m>0$.  
Then, with $s$ as in~\eqref{eq:sl2}, the Pitman ARE  
of $L_a$ with respect to $L_b$ is given by $\; {\rm eff}_{L_a,L_b}\pit=s(a)^2/s(b)^2$.
Further, within this family, $L_a$ maximizes the slope and thus the Pitman ARE 
with respect to the other members if and only if $a\in S(m,\Xi)$ with $S(m,\Xi)$ as in~\eqref{eq:Sopt}.
\end{proposition}

We sketch some related results and extensions, leaving the details to the reader.

First, it follows from~\eqref{eq:back}  that, with $a := \frac{1}{3}(3,1,1,-1,-1,-3)\tr$,
$K_2(\pi) = L_a(\pi)$ if $|\pi|\ge 3$, 
and~\eqref{eq:sl2} leads to $1/9$ again (indeed, $a=\frac{4}{9}e_1+\frac{1}{9}e_5$). This 
means that, as far as
$L_a$ tests with $a\tr m>0$ and Pitman efficiencies are concerned, there is 
no gain in using patterns of length three instead of length two if testing $\theta=0$ against $\theta > 0$. 

Secondly, the component ranks $(Q_n,R_n)$ for a sample of size $n$ from an FGM copula 
generate a parametric family $\mathscr{P} =\{P_\theta:\, -1\le \theta\le 1\}$ of distributions 
on $\bS_n\times\bS_n$.
We know from Section~\ref{sec:cop} that $\Pi_n= R_n\circ Q_n^{-1}$. Also, $\Pi_n$ is
a sufficient statistic for $\mathscr{P}$. 
Classical finite-sample methods lead to a locally most powerful rank test of exact size $\alpha$ for $H_0:\theta=0$ against
$H_1:\theta>0$, known as \emph{Spearman's one-sided rank correlation test}.
The test statistic, \emph{Spearman's rank
correlation}, can be expressed as $\frac{3}{n+1}\,K_2(\Pi_n)\, + \,\frac{n-2}{n+1}\,L_{\breve{a}}(\Pi_n)$
with $\breve{a}:=(1,1,1,-1,-1,-1)\tr = \frac{2}{3} e_1-\frac{1}{3}e_5\in S(m,\Xi)$. It follows from this
that the test has asymptotic slope $s(\breve{a})$ and Pitman ARE 1 with respect to each $L_a$ test 
with $a\in S(m,\Xi)$, especially Kendall's rank correlation test. In particular, only patterns of length three are needed.
We point out that the coincidence of the Pitman AREs of the rank correlation tests of Spearman and
Kendall for FGM as well as for various other parametric subclasses of alternatives is well-known; see, e.g. \cite{Pinelis}.
In \cite{GenestVerret} locally most powerful rank tests of independence for general copula models are developed, 
and with respect to these, Pitman AREs of competing rank tests of independence are derived.

Further we might use the tests in Proposition~\ref{prop:FGM3} 
in connection with the extended hypothesis $H_0:\theta\le 0$. 
Then for a given $\alpha\in (0,1)$ conditions on the vector $a$ that ensure that
\begin{equation}\label{eq:compoundFGM}
 \lim_{n\to\infty} P_\theta( L_a(\Pi_n) > c_{a;\alpha,n})=0\;\text{ for all }\theta<0,\text{ and }
  \lim_{n\to\infty} P_\theta(L_a(\Pi_n) > c_{a;\alpha,n})=1\;\text{ for all }\theta>0
\end{equation}
are of interest. In contrast to the efficiency considered above this is a global property.
The consistency of $U$-statistics, together with  
$E_\theta L_a(\Pi_n) = \sum_{i=1}^6 a_iC^\bS_\theta(\sigma_i)=\mu_a(\theta)$
and the values of $C^\bS_\theta$ given in~\eqref{eq:Cfunc1FGM}, imply 
that~\eqref{eq:compoundFGM} would follow if $\mu_a(\theta)>\mu_a(0)$ for $\theta>0$ and
$\mu_a(\theta)<\mu_a(0)$ for $\theta<0$, where 
$\mu_a(\theta)=a\tr e+\theta\, a\tr m +\frac{\theta^2}{200}\,a\tr c$,
with $e=(1,1,\ldots,1)\tr$ and $c=(2,-1,-1,-1,-1,2)\tr$. In view of  $c\tr e_5=c\tr e_6=0$ this holds
for all elements of $S(m,\Xi)$.

Finally, for the two-sided problem we can choose some $a\in \bR^6$ with
$m\tr a\neq 0$ and consider the test based on the test statistic
$M_a(\Pi_n):=\bigl |L_a(\Pi_n)-\mu_a\bigr|$ that rejects the the hypothesis if
$M_a(\Pi_n)>d_{a;\alpha,n}$, with $d_{a;\alpha,n}$ as upper $\alpha$ quantile of $M_a(\Pi_n)$
in the case where the hypothesis is true. 
Then an analogue  
of~\eqref{eq:aslocpower} can be obtained, and for two competing tests of this type, 
say with $a=a_1$ and $a=a_2$, the Pitman ARE of the latter with respect to the former 
may be defined as the ratio $s(a_2)^2/s(a_1)^2$; see~\cite[p.286--287]{WiMuF}.
However,  the assumption $m\tr a\neq 0$ 
excludes linear pattern statistics that may be of interest in the two-sided context. In the present
setting $m\tr a = 0$ implies that the vector $a$ is a linear combination of $e_5$ and $e_6$,
where $e_6$ leads to $L_{e_6}(\Pi_n)\equiv 1$. We then also have $v(a)=0$, so that the kernel
of the associated $U$-statistics is degenerate. Such kernels are outside the asymptotic normality framework
considered here.
For $a=e_5$ the (non-normal) limit distribution of
the test statistic is given in \cite[p.\,128]{JNZ}. This aspect is also important in the context of the BDY test
that is included in the simulation study of Subsection \ref{subsec:sim}. We refer to~\cite{NWD} for the limiting
null distribution of the test statistic, and to \cite{DDB} for the power performance under contiguous alternatives.

\subsection{Delay copulas with exponential delay distributions}\label{subsec:delaytest}
In this section 
$\mathscr{C}$ is the family of delay copulas $C_\theta=C(G_\theta)$ with exponential delay 
distributions ${\rm Exp}(\theta)$, $\theta>0$, so that $G_\theta(x)=1-e^{-\theta x}$, $x\ge 0$.
For a given $\theta_0>0$ 
we consider testing the hypothesis $H_0:\theta\ge \theta_0$ against the alternative $H_{-}:\theta<\theta_0$.
The number $I_n=t(\Pi_n,21)$ of inversions in $\Pi_n$ is an unbiased and consistent
estimator of the probability  
$\phi_I(\theta) :=  P_\theta(\Pi_{2}=21) = (e^{-\theta} - 1 +\theta)/\theta^2$,
where $\phi_I$ is continuous and strictly decreasing with negative derivative
$\phi_I^\prime(\theta)=-\bigl((2+\theta)e^{-\theta} - (2-\theta)\bigr)/\theta^3$,
$\theta> 0$, and limits $\lim_{\theta\to 0}\phi_I(\theta)=1/2$,
$\lim_{\theta\to\infty}\phi_I(\theta)=0$.  Thus the testing problem 
can equivalently be expressed as testing $\phi_I(\theta)\le \phi_I(\theta_0)$ 
against $\phi_I(\theta)>\phi_I(\theta_0)$.
For a given significance level $\alpha\in (0,1)$, we consider the \emph{$I$-test}
that rejects the hypothesis if $I_n > i_{n,\alpha}$, with
$i_{n,\alpha}=\inf\{t\in \bR : P_{\theta_0}(I_n > t )\le \alpha\}$ as the upper 
$\alpha$ quantile of the distribution of $I_n$ if $\theta=\theta_0$. Using the fact that 
$\{\Exp(\theta):\theta>0\}$ is a scale family of distributions it is not 
difficult to show that $\theta \mapsto P_{\theta}(I_n > t)$ is monotone non-increasing for all
$t\in\bR$, hence the $I$-test has a monotone non-increasing power function and the procedure keeps 
the given level $\alpha$ on the full hypothesis $\theta\ge \theta_0$.

If observations of the delays $D_1,\ldots,D_n$ are available, then 
there exists a uniformly most powerful test at level $\alpha$ for the testing problem.
This test, which we call the \emph{$D$-test}, rejects the hypothesis if 
$\bar{D}_n:=\frac{1}{n}\sum_{i=1}^nD_i \,>\, d_{n,\alpha}$, 
where $d_{n,\alpha}$ is the upper $\alpha$ 
quantile of the gamma distribution $G(n,n\theta_0)$ 
with density $n^n\theta_0^n x^{n-1}\,e^{-n\theta_0 x}/\Gamma(n)$, $x> 0$. Its power function 
$\theta \mapsto P_\theta (\bar{D}_n > d_{n,\alpha})$ is strictly decreasing. 

We now aim to compare the asymptotic performances of these tests, 
using the asymptotic relative efficiency (ARE) concepts introduced  by
Pitman and Bahadur respectively. We give a brief description of
these standard approaches on the basis of the testing problem posed
here. For a general, detailed discussion we refer to~\cite{Bahadur},~\cite{Nikitin} and~\cite{vdV}. 

Let $I=(I_n)_{n\in\bN}$ and $D=(\bar{D}_n)_{n\in\bN}$ be the sequences of test statistics
$I_n$ and $\bar{D}_n$, where the subscript $n$ refers to the size of the data.
ARE enables the comparison of the sequences of tests associated with $I$
and $D$. For $\alpha\in(0,1)$, $\beta\in(\alpha,1)$ and $\theta<\theta_0$
let $N_I(\alpha,\beta,\theta)$ be the smallest integer $m$ such that, for each $n\ge m$, 
the test rejecting the hypothesis if $I_n > i_{n,\alpha}$ has power at least $\beta$ at the alternative
$\theta$, and let $N_D$ be similarly defined. Then the Pitman ARE and the Bahadur ARE of $I$
with respect to $D$ are given by 
\begin{equation*}
\lim_{\theta\to \theta_0} 
             \frac{N_D(\alpha,\beta,\theta)}{N_I(\alpha,\beta,\theta)} \ \text{ and }\ 
    \lim_{\alpha\downarrow 0} \frac{N_D(\alpha,\beta,\theta)}{N_I(\alpha,\beta,\theta)}
\end{equation*}
respectively, provided that the limits exist.
In the first of these, the alternative approaches the hypothesis whereas in the second 
the alternative remains fixed, but we get a local version of the latter by letting $\theta$ tend to~$\theta_0$.

Beginning with the Pitman efficiency we introduce the functions
$\phi_D(\theta):=E_\theta(D_1)=1/\theta$ and 
$v_D(\theta):=\var_\theta(D_1)=1/\theta^2$. The first of these is differentiable 
at $\theta_0$ with $\phi_D^\prime(\theta_0)=-1/\theta_0^2<0$, 
the second is continuous and  positive at $\theta_0$, with $v_D(\theta_0)=1/\theta_0^2$.
Hence the  Lindeberg--Feller central limit theorem gives
\begin{equation*}
    \cL\left (\sqrt{n}\left (\bar{D}_n - \phi_D(\theta_n)\right )\big |\theta_n\right )
         \, \toweak  N\bigl(0,v_D(\theta_0)\bigr)
\end{equation*}
whenever $\theta_n\to\theta_0$. For the frequencies underlying the $I$-test Theorem~\ref{thm:patternCLT}
leads to asymptotic normality  with variance $v_I(\theta_0)$,
where
\begin{equation}\label{eq:var:Kn}
        v_I(\theta) \; :=\; \frac{2}{3\theta^4}\bigl(2\theta^2-3\theta-6
                            \,+\, 2(3\theta^2+2\theta+6)e^{-\theta}\, 
                            -\, (\theta+6)e^{-2\theta}\bigr);
\end{equation}
see \cite[Proof of Theorem 10]{BaGrDelDat}.                          
We now apply ~\cite[Theorem 14.19]{vdV} to obtain 
that the Pitman ARE  of the $I$-test with respect to the $D$-test is given by
\begin{equation*}
  {\rm eff}_{I,D}\pit\,=\,\frac{\phi_I^\prime(\theta_0)^2 \,v_D(\theta_0)}
                                                      {\phi_D^\prime(\theta_0)^2\, v_I(\theta_0)}
                                \,=\,   \frac{\theta_0^2 \,\phi_I^\prime(\theta_0)^2}{v_I(\theta_0)}.
\end{equation*}  

We approach the  Bahadur efficiency via the  
\emph{Bahadur exact slopes} of the individual  sequences of test statistics under 
consideration. As the functions $\theta\to P_\theta(\bar{D}_n \ge t)$ and
$\theta\to P_\theta(I_n \ge t)$ are both monotone non-increasing for all $t\in\bR$ and $n\in\bN$,  
the \emph{levels attained} by the $D$-test and the $I$-test are
\begin{equation*}
  L_{D,n}:= P_{\theta_0}(\bar{D}_n \ge t)|_{t=\bar{D}_n}\quad\text{and}
         \quad L_{I,n}:= P_{\theta_0}(I_n \ge t)|_{t=I_n}.
\end{equation*}
For $\theta < \theta_0$, the Bahadur exact slopes of the $D$- and 
$I$-sequences are the limits in $P_\theta$-probability, if they exist, 
of~\,$-\frac{2}{n}\log L_{D,n}$~\,and~\,$-\frac{2}{n}\log L_{I,n}$~\,respectively.

By the Cram\'er--Chernoff theorem, see, e.g.~\cite[Proposition 14.23]{vdV}, we have
\begin{equation*}\label{eq:G:Cramer}
  - \lim_{n\to\infty}\,\frac{1}{n} \log P_{\theta_0}(\bar{D}_n \ge t)\, 
                                     = \, \theta_0\,t - 1 - \log\,(\theta_0\,t) \quad\text{for all}~t>1/\theta_0.
\end{equation*}
For $\theta<\theta_0$ we use the weak law of large numbers to obtain that  
$\bar{D}_n \to  1/\theta$ in $P_\theta$-probability. 
Together with  the continuity of the function~$u\mapsto u - 1 - \log u$ on the positive half-line
it now follows that, for all $\theta < \theta_0$,
\begin{equation*}
    - \lim_{n\to\infty}\frac{1}{n}\log L_{D,n} \,
                   =\,\frac{\theta_0}{\theta} - 1 - \log \frac{\theta_0}{\theta}
                   =: c_D(\theta)
\end{equation*}
in $P_\theta$-probability; see also \cite[Theorem 7.2]{Bahadur} or \cite[Theorem 1.2.2]{Nikitin}.

In order to arrive at the corresponding limit for the $I$-sequence we use~\cite[Theorem 2.3]{NikiPoni} 
to obtain an $0<\epsilon<\frac{1}{2}-\phi_I(\theta_0)$ and an analytic function
$b:(-\epsilon,\epsilon)\longrightarrow \bR$, represented by the powers series
$b(t)=\sum_{j=2}^\infty  b_ j t^j$, $ t\in (-\epsilon,\epsilon)$, 
with $b_2=\frac{1}{2 v_I(\theta_0)}$ such that
\begin{equation*}
  - \lim_{n\to\infty}\,\frac{1}{n}\log P_{\theta_0}(I_n-\phi_I(\theta_0) \ge  t + o(1) )\,
             =\,\frac{1}{2 v_I(\theta_0)} t^2 + \sum_{j=3}^\infty b_j t^j
                  \quad\text{for all}~t\in (0,\epsilon).
\end{equation*}
Together with  $I_n \to \phi_I(\theta)$ in $P_\theta$-probability 
for all~$\theta<\theta_0$ we get that for 
$\phi_I^{-1}\left (\phi_I(\theta_0)+\epsilon\right ) < \theta < \theta_0$,
\begin{align*}
  -  \lim_{n\to\infty} \frac{1}{n}\log L_{I,n} = c_I(\theta)\quad\text{in } P_\theta \text{-probability},
\end{align*}
with
\begin{align*}
   c_I(\theta):= b\bigl(\phi_I(\theta)-\phi_I(\theta_0)\bigr)\ 
                           =\ \frac{\left (\phi_I(\theta)-\phi_I(\theta_0)\right )^2}{2v_I(\theta_0)}
                                    + \sum_{j=3}^\infty b_j\bigl(\phi_I(\theta)-\phi_I(\theta_0)\bigr)^j.
\end{align*}
The relation of $N_D(\alpha,\beta,\theta)$ and $N_I(\alpha,\beta,\theta)$ 
to $c_D(\theta)$ and $c_I(\theta)$ is as follows. 
Obviously, $L_{D,n}<\alpha$ if and only if $\bar{D}_n>d_{n,\alpha}$. Likewise,
with $S_n:=\{s\in\bR: P_{\theta_0}(I_n=s)>0\}$ the finite support of the distribution $\mathcal{L}(I_n|\theta_0)$ we have that
$i_{n,\alpha}\in S_n$, and that 
$L_{I,n}\le \alpha$ if and only if $I_n>i_{n,\alpha}$. Then arguing as in
\cite[Proof of Theorem 7.1]{Bahadur} or \cite[Proof of Theorem 14.22]{vdV}
we obtain  
\begin{equation*}
 N_D(\alpha,\beta,\theta) \sim - \frac{\log \alpha}{c_D(\theta)},\quad
 N_I(\alpha,\beta,\theta) \sim - \frac{\log \alpha}{c_I(\theta)}\quad \text{as } \alpha\downarrow 0,
\end{equation*}  
and deduce from this that
the Bahadur asymptotic relative efficiency of the $I$-test with respect to the $D$-test at any $\theta$
with $\phi_I^{-1}\bigl(\phi_I(\theta_0)+\epsilon\bigr)< \theta <\theta_0\,$
is given by $ {\rm eff}_{I,D}\ba(\theta)=\frac{c_I(\theta)}{c_D(\theta)} $.
Using
\begin{align*}
    \frac{\theta_0}{\theta} -1 - \log \frac{\theta_0}{\theta}       
                  \ &=\ \frac{1}{2\theta_0^2} (\theta -\theta_0)^2+o\bigl((\theta -\theta_0)^2\bigr),\\
    \phi_I(\theta)-\phi_I(\theta_0)
                   \ &=\ \phi_I^\prime(\theta_0) (\theta -\theta_0)+o(\theta -\theta_0)
\end{align*}  
as $\theta\to \theta_0$
%
together with the continuity of $\theta\mapsto v_I(\theta)$, we finally obtain
\begin{equation*}
  {\rm eff}_{I,D}\ba\,=\,\frac{\theta_0^2\,\phi_I^\prime(\theta_0)^2}{v_I(\theta_0)}.
\end{equation*}
Taken together this proves the following result.

\begin{proposition}\label{prop:delexp} 
Let $\theta_0>0$ be given and consider the problem of testing $H_0: \theta\ge\theta_0$ 
against $H_-:\theta<\theta_0$ 
in the delay model with exponential delay distribution $\Exp(\theta)$.
In this situation, the Pitman asymptotic relative efficiency
and the local Bahadur asymptotic relative efficiency of the $I$-test with respect to the $D$-test are equal, 
and  given by $\theta_0^2 \phi_I'(\theta_0)^2/v_I(\theta_0)$.
\end{proposition}

A third concept, \emph{local Hodges-Lehmann ARE}, can similarly be analyzed and leads to the same value
in the situation considered here. 
The (lengthy) proof requires some additional technical arguments and will be part of a separate,
more general project.  

Also, similar statements can be derived for the related one sided testing problem
$H_0:\theta \le \theta_0$ against $H_+:\theta >\theta_0$, and the 
two-sided testing problem $H_0:\theta = \theta_0$ against $H_1:\theta \neq \theta_0$. 
Finally, we refer to \cite{KallenbergLedwina}, \cite{Kourouklis} and \cite{Wieand} for general results
on the coincidence of  different efficiency approaches.

\section{Implementation and simulations}\label{sec:implementation}

In this section we consider computational aspects, and
we provide an experimental comparison of our procedures with traditional tests
in the important special case of testing for independence.

\subsection{Exact enumeration and Monte Carlo approximation}\label{subsec:enumerationMC}

Finding the number of occurrences of a given pattern of length $k$ in a permutation
of size $n$ would require about $n^k$ steps in a naive procedure.
The search for faster algorithms has brought forth a variety of different
techniques. In particular, in \cite{EZLeng} the concept of corner trees,
together with a specific binary tree data structure, is used to
obtain counting algorithms for patterns of length $k=3$ that are roughly linear in $n$,
and for $k=4$ an algorithm with rough order $n^{3/2}$ is given. Further, there are theoretical 
considerations that lead to lower bounds that rule out fast algorithms for  $k$ 
and $n$.

This problem also arises in the (more general) context of $U$-statistics, where for a kernel
$h$ of order $k$ and a data set of size $n$ the value of $h$ is required for all 
$k$-element subsets of the data set. The use of a large number $N$ of randomly chosen
subsets as an alternative to exact enumeration has been investigated in~\cite{JansonMC},
with the remarkable result that, in the non-degenerate case of present interest
and for each fixed $k$, there ``is no point in taking $N$ much larger than~$n$''.  

\subsection{A simulation study for the case of testing for independence} \label{subsec:sim}

Based on the independent variables $Z_i=(X_i,Y_i)$, $i=1,\ldots,n$, 
we consider testing independence of $X_i$ and $Y_i$ against the 
general alternative that $X_i$ and $Y_i$ are not independent. With $C_0$
the independence copula, i.e. $C_0(u,v)\,=\,uv$ for all $u,v\in [0,1]$, this leads to  
the goodness-of-fit problem discussed in Subsection \ref{subsec:CvMGF}.
Clearly for $C=C_0$ the nondegeneracy  condition \eqref{eq:nondeg} is satisfied, hence the
pattern-based Cram\'er--von Mises and the Kolmogorov--Smirnov test are both consistent.
This even holds for the truncated versions with truncation parameter
$k_n\equiv 4$, with test statistics ${\rm CvM}_{n,4}\gf$ or ${\rm KS}_{n,4}\gf$,
as $C^\bS(\sigma) = 1/4!$ for all $\sigma\in\bS_4$
if and only if $C$ is the independence copula; see~\cite[Proposition 9]{Yana}. 
This result further  implies that the corresponding tests 
that use patterns of length four only, with test statistics
${\rm CvM^*}_{\hspace*{-1.5mm}n,4}\gf
       :=n\sum_{\sigma\in\bS_4} \left (\,T_n(\sigma) - \nf{1}{24}\,\right )^2$ 
and
${\rm KS^*}_{\hspace*{-1.5mm}n,4}\gf
       :=\sqrt{n} \max_{\sigma\in\bS_4}\, \left |\,T_n(\sigma) - \nf{1}{24}\,\right |$,
are also consistent.  

\begin{table}[!]
  \caption{Empirical power values of ${\rm CvM}_{n,4}\gf$, ${\rm KS}_{n,4}\gf$, ${\rm HBKR}_n^*$
    (first line in each row), and empirical power values of
    ${\rm CvM^*}_{\hspace*{-1.5mm}n,4}\gf$, ${\rm KS^*}_{\hspace*{-1.5mm}n,4}\gf$, ${\rm BDY}_n$
    (second line in each row) for FGM  and Clayton alternatives}
    \label{tab:eff}
  \vspace{1mm}
  \begin{tabular}{cccccccccccc}
    \noalign{\vspace{1mm}}
                                   & \multicolumn{3}{c}{${\rm CvM}_{n,4}\gf/{\rm CvM^*}_{\hspace*{-1.1mm}n,4}\gf$} && \multicolumn{3}{c}{${\rm KS}_{n,4}\gf/{\rm KS^*}_{\hspace*{-1.1mm}n,4}\gf$} && \multicolumn{3}{c}{${\rm HBKR}_n^*/{\rm BDY}_n$} \\
    \noalign{\vspace{1mm}}
      $\alpha$                           & 0.1 & 0.05 & 0.025 && 0.1 & 0.05 & 0.025 && 0.1 & 0.05 & 0.025 \\
    \noalign{\vspace{1mm}}
    \hline
    \noalign{\vspace{1.5mm}}
    \multicolumn{12}{c}{$=\joinrel=\joinrel=\joinrel=$\ \ FGM\ \ $=\joinrel=\joinrel=\joinrel=$} \\
    \noalign{\vspace{1mm}}
    $\theta = \nf{1}{4}$ & \mc{0.15}{0.15} &\mc{0.08}{0.08} & \mc{0.05}{0.05} && \mc{0.15}{0.14} & 
               \mc{0.08}{0.08} & \mc{0.05}{0.04} && \mc{0.14}{0.15} & \mc{0.08}{0.08} & \mc{0.05}{0.05} \\
    \noalign{\vspace{2mm}}
    $\theta = \nf{1}{2}$  & \mc{0.30}{0.29} & \mc{0.20}{0.19} & \mc{0.13}{0.12} && \mc{0.30}{0.26} 
    & \mc{0.20}{0.16} & \mc{0.13}{0.10} && \mc{0.28}{0.29} & \mc{0.19}{0.20} & \mc{0.12}{0.13} \\
    \noalign{\vspace{2mm}}
    $\theta = 1$          & \mc{0.75}{0.72} & \mc{0.64}{0.61} & \mc{0.52}{0.50} && \mc{0.75}{0.65} 
    & \mc{0.63}{0.52} & \mc{0.52}{0.40} && \mc{0.71}{0.73} & \mc{0.59}{0.62} & \mc{0.47}{0.50} \\            
    \noalign{\vspace{1.5mm}}
 \multicolumn{12}{c}{$=\joinrel=\joinrel=\joinrel=$\ \  Clayton\ \ $=\joinrel=\joinrel=\joinrel=$} \\
    \noalign{\vspace{1mm}}
      $\kappa = -\nf{1}{4}$  & \mc{0.44}{0.44} & \mc{0.32}{0.32} & \mc{0.22}{0.22} && \mc{0.44}{0.38} &
           \mc{0.32}{0.26}  & \mc{0.22}{0.18} && \mc{0.41}{0.41} & \mc{0.29}{0.29} & \mc{0.19}{0.20} \\
   \noalign{\vspace{2mm}}
      $\kappa = \nf{1}{4}$  & \mc{0.31}{0.32} & \mc{0.21}{0.22} & \mc{0.14}{0.15} && \mc{0.30}{0.30} 
            & \mc{0.20}{0.20} & \mc{0.14}{0.13} && \mc{0.28}{0.29} & \mc{0.19}{0.19} & \mc{0.12}{0.12} \\
    \noalign{\vspace{2mm}}
      $\kappa = \nf{1}{2}$  & \mc{0.66}{0.68} & \mc{0.54}{0.57} & \mc{0.42}{0.46} && \mc{0.66}{0.65} & 
               \mc{0.53}{0.54} & \mc{0.42}{0.43} && \mc{0.62}{0.62} & \mc{0.50}{0.50} & \mc{0.38}{0.39}         
  \end{tabular}

  \medskip  
 (a) Sample size $n=50$

  \vspace{0.3cm} 
  \begin{tabular}{ccccccccccccc}
    \noalign{\vspace{1mm}} 
          & \multicolumn{3}{c}{${\rm CvM}_{n,4}\gf/{\rm CvM^*}_{\hspace*{-1.1mm}n,4}\gf$} && \multicolumn{3}{c}{${\rm KS}_{n,4}\gf/{\rm KS^*}_{\hspace*{-1.1mm}n,4}\gf$} && \multicolumn{3}{c}{${\rm HBKR}_n^*/{\rm BDY}_n$} \\
    \noalign{\vspace{1mm}}
      $\alpha$                           & 0.1 & 0.05 & 0.025 && 0.1 & 0.05 & 0.025 && 0.1 & 0.05 & 0.025 \\
    \noalign{\vspace{1mm}}
    \hline 
    \noalign{\vspace{1.5mm}}
    \multicolumn{12}{c}{$=\joinrel=\joinrel=\joinrel=$\ \ FGM\ \ $=\joinrel=\joinrel=\joinrel=$} \\
    \noalign{\vspace{1mm}}
   $\theta = \nf{1}{4}$ & \mc{0.21}{0.20} & \mc{0.13}{0.12} & \mc{0.08}{0.08} && \mc{0.21}{0.18} &
 \mc{0.13}{0.11} & \mc{0.08}{0.07} && \mc{0.19}{0.20} & \mc{0.12}{0.12} & \mc{0.08}{0.08} \\
    \noalign{\vspace{2mm}}
   $\theta = \nf{1}{2}$           & \mc{0.51}{0.48} & \mc{0.38}{0.36} & \mc{0.28}{0.27} && \mc{0.51}{0.43} & 
\mc{0.38}{0.31} & \mc{0.28}{0.23} && \mc{0.48}{0.49} & \mc{0.36}{0.36} & \mc{0.26}{0.27} \\
    \noalign{\vspace{2mm}}
   $\theta = 1$                   & \mc{0.96}{0.95} & \mc{0.93}{0.92} & \mc{0.88}{0.86} && \mc{0.96}{0.92} & 
    \mc{0.93}{0.86} & \mc{0.88}{0.78} && \mc{0.95}{0.95} & \mc{0.91}{0.91} & \mc{0.85}{0.86} \\            
    \noalign{\vspace{1.5mm}}    
 \multicolumn{12}{c}{$=\joinrel=\joinrel=\joinrel=$\ \  Clayton\ \ $=\joinrel=\joinrel=\joinrel=$} \\
    \noalign{\vspace{1mm}}
     $\kappa = -\nf{1}{4}$  & \mc{0.68}{0.72} & \mc{0.56}{0.60} & \mc{0.45}{0.47} && \mc{0.68}{0.61} & 
    \mc{0.56}{0.48} & \mc{0.45}{0.37} && \mc{0.65}{0.66} & \mc{0.53}{0.53} & \mc{0.42}{0.42} \\
    \noalign{\vspace{2mm}}
     $\kappa = \nf{1}{4}$           & \mc{0.51}{0.53} & \mc{0.38}{0.40} & \mc{0.28}{0.31} && \mc{0.50}{0.51} & 
    \mc{0.38}{0.38} & \mc{0.28}{0.29} &&  \mc{0.47}{0.47} & \mc{0.35}{0.35} & \mc{0.26}{0.25} \\
    \noalign{\vspace{2mm}}
    $\kappa = \nf{1}{2}$           & \mc{0.90}{0.93} & \mc{0.84}{0.88} & \mc{0.76}{0.81} && \mc{0.90}{0.91} &
   \mc{0.84}{0.86} & \mc{0.76}{0.79} &&  \mc{0.89}{0.89} & \mc{0.82}{0.82} & \mc{0.74}{0.73}          
\end{tabular}

  \medskip  
  (b) Sample size $n=100$
  \vspace*{-0.48cm}
\end{table}

We provide a small simulation study, comparing the power performance of the new tests
based on
${\rm CvM}_{n,4}\gf,~{\rm KS}_{n,4}\gf,~{\rm CvM^*}_{\hspace*{-1.5mm}n,4}\gf,~{\rm KS^*}_{\hspace*{-1.5mm}n,4}\gf$
to that of the Bergsma--Dassios–Yan\-agimoto (BDY) test, see \cite{BergsmaDassios} and \cite{ShiDrtonHan}, 
and that of the ${\rm HBKR}^*$ test, a variant of the Hoeffding-Blum-Kiefer-Rosenblatt (HBKR)
independence test, see~\cite{BlumKieferRosenblatt} and~\cite{Hoeffding}, introduced and discussed 
in~\cite{Baringhaus}. The test statistics of the latter are the linear pattern statistic of degree four, 
\begin{equation*}
  {\rm BDY}_n=n\left (\ \frac{2}{3}\sum_{\sigma \in \mathsf C} T_n(\sigma) - \frac{1}{3}\sum_{\sigma \in \mathsf D} T_n(\sigma)\right ),
\end{equation*}  
with subsets $\mathsf C = \{1234,1243,2134,2143,3412,3421,4312,4321\}$, $\mathsf D := \bS_4\setminus \mathsf C$
of $\bS_4$, and
\begin{equation*}
  {\rm HBKR}_n^*=\frac{1}{n^4}\sum_{j=1}^n \left (m_1(j)m_4(j)\,-\,m_2(j)m_3(j)\right )^2,
\end{equation*}  
where $m_1(j),m_2(j),m_3(j),m_4(j)$ are the number of points $(i,\Pi_n(i))$, $i=1,\dots,n$, 
in the regions
$\{(x,y)\in \bR^2: x<j, y<\Pi_n(j)\}$, $\{(x,y)\in \bR^2: x>j, y<\Pi_n(j)\}$, $\{(x,y)\in \bR^2: x<j, y>\Pi_n(j)\}$,
$\{(x,y)\in \bR^2: x>j, y>\Pi_n(j)\}$.
By design, all these tests are invariant under strictly increasing or
strictly decreasing transformations of the coordinate variables. We mention in passing that
the test statistics can be calculated easily.  
As alternative distributions for the $Z$-variables, FGM copulas with dependence parameter 
$\theta\in (0,1]$, see \eqref{eq:defFGM}, and elements of the 
Clayton copula family given by
\begin{align*}
  C_\kappa^{\rm Clayton}(u,v)\,=\left (\max\left (u^{-\kappa}+v^{-\kappa} - 1,0\right )\right )^{-\nf{1}{\kappa}},\,
  \quad (u,v)\in [0,1]^2,
\end{align*}
with dependence parameter $\kappa\in [-1,\infty)\setminus \{0\}$, are chosen.
In view of the symmetry properties given at the beginning of 
Section~\ref{subsec:FGMtest} it is enough to consider positive values of $\theta$ in the FGM case.

Table \ref{tab:eff} shows the empirical power values for significance level
$\alpha\in \{0.1,0.05,0.025\}$, sample size $n\in \{50,100\}$ 
and dependence parameters $\theta \in \{1,\nf{1}{2},\nf{1}{4}\}$, 
$\kappa\in \{-\nf{1}{4},\nf{1}{4},\nf{1}{2}\}$, obtained from simulations with
10000 replications. All results are rounded to two decimal places.
For the critical values, the empirical upper $\alpha$ quantiles 
of the respective test statistics in the case where the hypothesis is true, separate 
simulation with 100,000 replications have been used. 

From the values in Table \ref{tab:eff} we conclude that the six tests show roughly
the same performance. For ${\rm HBKR}$ and 
BDY this has already been observed earlier in \cite{ShiDrtonHan}, see their Table~2. Of course,
all these empirical findings refer to the specific alternative distributions taken into consideration,
and  to the special case of testing independence. 

\section*{Acknowledgments} We would like to thank the associate editor and the referees 
for their comments, which have led to several improvements of the paper.   

\bibliographystyle{alpha} 
{}

\end{document}